\documentclass[11pt,a4paper]{article}
\usepackage[latin1]{inputenc}
\usepackage{amsmath}
\usepackage{amsthm}
\usepackage{amsfonts}
\usepackage{amsfonts,amsthm,latexsym,amsmath,amssymb,amscd,epsfig,psfrag,enumerate}
\usepackage{graphics,graphicx, bezier, float, color, hyperref}
\usepackage{amssymb,url}
\usepackage{multienum}
\usepackage[table]{xcolor}
\usepackage{multicol,multirow}
\usepackage{graphicx}
\usepackage{fancyvrb}
\usepackage{parskip}
\usepackage[toc,page]{appendix}
\usepackage[top=2.7cm, bottom=2.7cm, left=1.5cm, right=1.5cm]{geometry}
\usepackage{xcolor}

\usepackage{blkarray}
\newtheorem{theorem}{Theorem}[section]
\newtheorem{lemma}{Lemma}[section]

\usepackage[none]{hyphenat}[section]
\newtheorem{definition}{Definition}[section]

\numberwithin{equation}{section}
\numberwithin{table}{section}
\numberwithin{figure}{section}

\title{On a nonlinear Diophantine equation with powers of
	three consecutive $k$--Lucas Numbers}
\author{Herbert Batte$^{1,*} $ and Florian Luca$^{1,2}$}
\date{}

\begin{document}
	\maketitle
	\abstract{ Let $(L_n^{(k)})_{n\geq 2-k}$ be the sequence of $k$--generalized Lucas numbers for some fixed integer $k\ge 2$ whose first $k$ terms are $0,\ldots,0,2,1$ and each term afterwards is the sum of the preceding $k$ terms. In this paper, we completely solve the nonlinear Diophantine equation $\left(L_{n+1}^{(k)}\right)^x+\left(L_{n}^{(k)}\right)^x-\left(L_{n-1}^{(k)}\right)^x=L_m^{(k)}$, in nonnegative integers $n$, $m$, $k$, $x$, with $k\ge 2$.} 
	
	
	
	\thanks{$ ^{*} $ Corresponding author}
	\section*{Statements and Declarations}
	
	\textbf{Competing Interests:} The authors declare that they have no conflict of interest.
	\\ \\
	\textbf{Ethical Approval:} This article does not contain any studies involving human participants or animals performed by any of the authors.
	\\ \\
	\textbf{Data Availability Statement:}
	Data sharing is not applicable to this article as no datasets were generated or analyzed during the current study.
	\\ \\
	\textbf{Author Contributions:} The authors contributed equally to the conceptualization, analysis, and writing of the manuscript.
		
	\section{Introduction}\label{intro}
	\subsection{Background}\label{sec:1.1}
	For any integer \( k \ge 2 \), the sequence of $k$--generalized Lucas numbers is defined by the recurrence relation:
	\[ L_n^{(k)} = L_{n-1}^{(k)} + \cdots + L_{n-k}^{(k)}, \quad \text{for all} \ n \ge 2, \]
	with the initial terms set as \( L_0^{(k)} = 2 \) and \( L_1^{(k)} = 1 \) for all \( k \ge 2 \), and \( L_{2-k}^{(k)} = \cdots = L_{-1}^{(k)} = 0 \) for \( k \ge 3 \). For \( k = 2 \), this reduces to the classical sequence of Lucas numbers, and the superscript ${}^{(k)}$ is typically omitted.
	
	The study of sequences such as the Fibonacci and Lucas numbers, and their generalizations, has led to the discovery of numerous interesting identities and Diophantine equations. Similar to how various authors have explored identities involving the $k$--Fibonacci numbers, the $k$--Lucas numbers also exhibit a rich structure that can be investigated through the lens of number theory. In \cite{GGL}, Gómez et al. studied the Diophantine equation 
	\begin{align}\label{eqGGL}
		\left(F_{n+1}^{(k)}\right)^x+\left(F_{n}^{(k)}\right)^x-\left(F_{n-1}^{(k)}\right)^x=F_m^{(k)},
	\end{align}
	in nonnegative integers $n$, $m$, $k$, $x$, with $n\ge 1$ and $k\ge 2$, where $F_r^{(k)}$ is the $r^{\text{th}}$ term of the $k$--Fibonacci sequence. They showed that the Diophantine eq. \eqref{eqGGL} does not have non--trivial solutions $(k, n, m, x)$ with $k \ge 2$, $n \ge 1$, $m \ge 2$ and $x \ge 1$.
	
	Motivated by the work in \cite{GGL}, we study the same Diophantine equation but with $k$--Lucas numbers. Specifically, we solve the nonlinear Diophantine equation 
	\begin{align}\label{eq:main}
		\left(L_{n+1}^{(k)}\right)^x+\left(L_{n}^{(k)}\right)^x-\left(L_{n-1}^{(k)}\right)^x=L_m^{(k)},
	\end{align}
	in nonnegative integers $n$, $m$, $k$, $x$, with $k\ge 2$. We prove the following result.
	\subsection{Main Result}\label{sec:1.2}
	\begin{theorem}\label{thm1.1} 
		Let $(L_n^{(k)})_{n\geq 2-k}$ be the sequence of $k$--generalized Lucas numbers. Then, the only integer solutions $(n,m,k,x)$ to Eq. \eqref{eq:main} with $n\ge 0,~m\ge 0,~k\ge 2$ and $x\ge 0$ (with the convention that $0^0:=1$ in case $(n,x)=(0,0)$ and $k\ge 3$) are
		\begin{align*}
			(n,m,k,x)\in \{(n,1,k,0),(1,0,k,1):k\ge 2\}\cup\{(0,2,k,1),(1,3,k,2):k\ge 3\}\cup \{(0,3,2,1),(0,3,2,2)\}.
		\end{align*}
	\end{theorem}
	
	\section{Methods}
	\subsection{Preliminaries}
	It is known that 
	\begin{align}\label{eq2.2}
		L_n^{(k)} = 3 \cdot 2^{n-2},\qquad \text{for all}\qquad 2 \le n \le k.
	\end{align}
	Additionally, $L_{k+1}^{(k)}=3\cdot 2^{k-1}-2$ and by induction one proves that 
	\begin{equation}
		\label{eq:32}
		L_n^{(k)}<3\cdot 2^{n-2}\qquad {\text{\rm holds for all}}\qquad n\ge k+1.
	\end{equation}
	Next, we revisit some properties of the $k$--generalized Lucas numbers. They form a linearly recurrent sequence of characteristic polynomial
	\[
	\Psi_k(x) = x^k - x^{k-1} - \cdots - x - 1,
	\]
	which is irreducible over $\mathbb{Q}[x]$. The polynomial $\Psi_k(x)$ possesses a unique real root $\alpha(k)>1$ and all the other roots are inside the unit circle, see \cite{MIL}. The root  $\alpha(k):=\alpha$ is in the interval
	\begin{align}\label{eq2.3}
		2(1 - 2^{-k} ) < \alpha < 2
	\end{align}
	as noted in \cite{WOL}. As in the classical case when $k=2$, it was shown in \cite{BRL} that 
	\begin{align}\label{eq2.4}
		\alpha^{n-1} \le L_n^{(k)}\le2\alpha^n, \quad \text{holds for all} \quad n\ge0, ~k\ge 2.
	\end{align}
	Also, $L_{-1}^{(k)}=-1$ if $k=2$ and $0$ otherwise. In particular, combining \eqref{eq2.4} and \eqref{eq:main}, we have, if $n\ge 1$, 
	\begin{align*}
		\alpha^{m-1} \le L_m^{(k)}=\left(L_{n+1}^{(k)}\right)^x+\left(L_{n}^{(k)}\right)^x-\left(L_{n-1}^{(k)}\right)^x<2\left(L_{n+1}^{(k)}\right)^x\le \alpha^{(n+3)x+2},
	\end{align*}
	where we have used the fact that $2<\alpha^2$, for all $k\ge 2$. The above inequality also holds for $n=0$ since in this case 
	$$
	(L_{n+1}^{(k)})^x+(L_n^{(k)})^x-(L_{n-1}^{(k)})^x\le 1+2^x+1\le 2^{x+1}<\alpha^{2x+2}<\alpha^{(n+3)x+2}.
	$$
	This gives $m<(n+3)x+3$. On the other hand, if $n\ge 1$,  and $x\ge 2$, we have
	\begin{align*}
		\alpha^{m+2}>2\alpha^{m} \ge L_m^{(k)}=\left(L_{n+1}^{(k)}\right)^x+\left(L_{n}^{(k)}\right)^x-\left(L_{n-1}^{(k)}\right)^x>\alpha^{nx}+\alpha^{(n-1)x}-2^x\alpha^{(n-1)x}\ge \alpha^{nx-1},
	\end{align*}
	from which we have $m>nx-3$. The above inequality also holds for $x=1$ since in this case 
	$$
	\alpha^{m+2}>2\alpha^m\ge L_m^{(k)}=L_{n+1}^{(k)}+L_n^{(k)}-L_{n-1}^{(k)}\ge 2L_{n}^{(k)}> \alpha^{n},\quad {\text{\rm so}}\qquad m>n-3.
	$$
	Finally, the above inequality holds for $n=0$ as well since we assume that $m\ge 0$. Therefore, we have
	\begin{align}\label{m_b}
		nx-3 <m<(n+3)x+3.
	\end{align}
	
	Let $k\ge 2$ and define
	\begin{equation}
		\label{eq:fk}
		f_k(x):=\dfrac{x-1}{2+(k+1)(x-2)}.
	\end{equation}
	We have 
	$$
	\frac{df_k(x)}{dx}=-\frac{(k-1)}{(2+(k+1)(x-2))^2}<0\qquad {\text{\rm for~all}}\qquad x>0.
	$$
	In particular, inequality \eqref{eq2.3} implies that
	\begin{align}\label{eq2.5}
		\dfrac{1}{2}=f_k(2)<f_k(\alpha)<f_k(2(1 - 2^{-k} ))\le \dfrac{3}{4},
	\end{align}
	for all $k\ge 3$. It is easy to check that the above inequality holds for $k=2$ as well. Further, it is easy to verify that $|f_k(\alpha_i)|<1$, for all $2\le i\le k$, where $\alpha_i$ are the remaining 
	roots of $\Psi_k(x)$ for $i=2,\ldots,k$.
	
	Moreover, it was shown in \cite{BRL} that
	\begin{align}\label{eq2.6}
		L_n^{(k)}=\displaystyle\sum_{i=1}^{k}(2\alpha_i-1)f_k(\alpha_i)\alpha_i^{n-1}~~\text{and}~~\left|L_n^{(k)}-f_k(\alpha)(2\alpha-1)\alpha^{n-1}\right|<\dfrac{3}{2},
	\end{align}
	for all $k\ge 2$ and $n\ge 2-k$. This means that 
	\begin{equation}
		\label{eq:Lnwitherror}
		L_n^{(k)}=f_k(\alpha)(2\alpha-1)\alpha^{n-1}+e_k(n), \qquad {\text{\rm where}}\qquad |e_k(n)|<1.5.
	\end{equation} 
	The left expression in \eqref{eq2.6} is known as the Binet formula for $L_n^{(k)}$. Furthermore, the right inequality in \eqref{eq2.6} shows that the contribution of the zeros that are inside the unit circle to $L_n^{(k)}$ is small. 
	
	Next, we state the following result which we shall use later in the proof of the main result. It is stated as Lemma 3 in \cite{Fay}.
	\begin{lemma}[Lemma 3 in \cite{Fay}]\label{fay}
		Let $k \ge 2$, $c \in (0,1)$ and $2\le n < 2^{ck}$. Then
		
		\[
		L_n^{(k)} = 3 \cdot 2^{n-2}(1 + \zeta'_n), \quad \text{with} \quad |\zeta'_n| < \begin{cases} 
			\dfrac{4}{2^{(1-c)k}}, & \text{if} \ c \le 0.693, \\
			\dfrac{8.1}{2^{(1-c)k}}, & \text{otherwise}.
		\end{cases}
		\]
		
	\end{lemma}
	
	Next, we prove the following result. 
	\begin{lemma}\label{lem:Lnx}
		Let \( k \ge 2 \), \( x \), \( n \) be positive integers. Then
		\[
		\left(L_n^{(k)}\right)^x = f_k(\alpha)^x (2\alpha-1)^x\alpha^{(n-1)x}(1 + \eta_n),
		\]
		with 
		\[
		|\eta_n| < \dfrac{1.5xe^{1.5x/\alpha^{n-1}}}{\alpha^{n-1}}.
		\]
	\end{lemma}
	\begin{proof}
		By the Binet formula in \eqref{eq:Lnwitherror} and inequalities \eqref{eq2.5}, it follows that
		\[
		\left(L_n^{(k)}\right)^x= f_k(\alpha)^x (2\alpha-1)^x\alpha^{(n-1)x} \left( 1 + \dfrac{e_k(n)}{f_k(\alpha)(2\alpha-1) \alpha^{n-1}} \right)^x.
		\]
		So, if we define
		\[
		\eta_n := \left( 1 + \dfrac{e_k(n)}{f_k(\alpha)(2\alpha-1) \alpha^{n-1}} \right)^x - 1 = \sum_{j=1}^x \binom{x}{j} \left( \frac{e_k(n)}{f_k(\alpha)(2\alpha-1) \alpha^{n-1}} \right)^j;
		\]
		we get
		\[
		|\eta_n| < \sum_{j=1}^x  \dfrac{(1.5x/\alpha^{n-1})^j}{j!} < \dfrac{1.5x}{\alpha^{n-1}} \sum_{j=1}^x \dfrac{(1.5x/\alpha^{n-1})^{j-1}}{(j-1)!} \le \dfrac{1.5xe^{1.5x/\alpha^{n-1}}}{\alpha^{n-1}},
		\]
		where we have used the fact $ |e_k(n)| < 1.5$ and $ (2\alpha-1)f_k(\alpha)>1$.
	\end{proof}
	
	Lastly here, we prove the following.
	\begin{lemma}\label{Ln:x}
		Let $x \geq 1$, $k \geq 2$, $i \in \{-1,0, 1\}$, $n + i \geq k + 2$ and $\max\{n + i, 16x\} < 2^{ck}$ for some $c \in (0, 0.25)$. Then, the estimate
		\[
		\left(L_{n+i}^{(k)}\right)^x = 3^x\cdot 2^{(n+i-2)x} \left(1 + \xi_{n+i} \right)
		\]
		holds with
		\[
		|\xi_{n+i}| <\dfrac{2}{2^{(1-2c)k}}.
		\]
	\end{lemma}
	\begin{proof}
		By Lemma \ref{fay} with $c \in (0, 0.25)$, we have
		\[
		L_{n+i}^{(k)} = 3 \cdot 2^{n+i-2} \left( 1 + \zeta'_{n+i} \right), \quad \text{with} \quad |\zeta'_{n+i}| < \dfrac{4}{2^{(1-c)k}}.
		\]
		Hence,
		\begin{eqnarray*}
			\left(L_{n+i}^{(k)}\right)^x  & = &  3^x\cdot 2^{(n+i-2)x} \left(1 +\zeta_{n+i}' \right)^x\\
			& = & 3^x 2^{(n+i-2)x} \exp(x\log(1+\zeta_{n+i}'))\\
			& = & 3^x 2^{(n+i-2)x} \exp(\eta_{n+i}')\qquad |\eta_{n+i}'|<2x|\zeta_{n+i}'|\\
			& = & 3^x2^{(n+i-2)x}(1+\xi_{n+i}')\qquad |\xi_{n+i}'|<2|\eta_{n+i}'|<4x|\zeta_{n+i}'|<\frac{16x}{2^{(1-c)k}}<\frac{1}{2^{(1-2c)k}}\left(<\frac{1}{2}\right)
		\end{eqnarray*}
		The above calculations are justified since both inequalities $|\log(1+y)|<2|y|$ and $|\exp y-1|<2|y|$ hold for $y\in (-1/2,1/2)$, and the fact that 
		all the intermediate quantities $\zeta_{n+i}',~\eta_{n+i}',\xi_{n+i'}$ are in the range $(-1/2,1/2)$ follows from the very last inequality above. 
	\end{proof}
	
	\subsection{Linear forms in logarithms}
	We use Baker--type lower bounds for nonzero linear forms in logarithms of algebraic numbers. There are many such bounds mentioned in the literature but we use one of Matveev from \cite{MAT}. Before we can formulate such inequalities, we need the notion of height of an algebraic number recalled below.  
	
	\begin{definition}\label{def2.1}
		Let $ \gamma $ be an algebraic number of degree $ d $ with minimal primitive polynomial over the integers $$ a_{0}x^{d}+a_{1}x^{d-1}+\cdots+a_{d}=a_{0}\prod_{i=1}^{d}(x-\gamma^{(i)}), $$ where the leading coefficient $ a_{0} $ is positive. Then, the logarithmic height of $ \gamma$ is given by $$ h(\gamma):= \dfrac{1}{d}\Big(\log a_{0}+\sum_{i=1}^{d}\log \max\{|\gamma^{(i)}|,1\} \Big). $$
	\end{definition}
	In particular, if $ \gamma$ is a rational number represented as $\gamma=p/q$ with coprime integers $p$ and $ q\ge 1$, then $ h(\gamma ) = \log \max\{|p|, q\} $. 
	The following properties of the logarithmic height function $ h(\cdot) $ will be used in the rest of the paper without further reference:
	\begin{equation}\nonumber
		\begin{aligned}
			h(\gamma_{1}\pm\gamma_{2}) &\leq h(\gamma_{1})+h(\gamma_{2})+\log 2;\\
			h(\gamma_{1}\gamma_{2}^{\pm 1} ) &\leq h(\gamma_{1})+h(\gamma_{2});\\
			h(\gamma^{s}) &= |s|h(\gamma)  \quad {\text{\rm valid for}}\quad s\in \mathbb{Z}.
		\end{aligned}
	\end{equation}
	In Section 3, equation (12) of \cite{Brl} these properties were used to show the following inequality:
	\begin{align}\label{eq2.9}
		h\left(f_k(\alpha)\right)<3\log k, ~~\text{for all}~~k\ge 2.
	\end{align}
	
	A linear form in logarithms is an expression
	\begin{equation}
		\label{eq:Lambda}
		\Lambda:=b_1\log \gamma_1+\cdots+b_t\log \gamma_t,
	\end{equation}
	where $\gamma_1,\ldots,\gamma_t$ are positive real  algebraic numbers and $b_1,\ldots,b_t$ are integers. We assume, $\Lambda\ne 0$. We need lower bounds 
	for $|\Lambda|$. We write ${\mathbb K}:={\mathbb Q}(\gamma_1,\ldots,\gamma_t)$ and $D$ for the degree of ${\mathbb K}$ over ${\mathbb Q}$.
	We give a direct consequence of Matveev's inequality from \cite{MAT}. That is, we quote it in a form which we shall use. 
	
	\begin{theorem}[Matveev, \cite{MAT}]
		\label{thm:Mat} 
		Put $\Gamma:=\gamma_1^{b_1}\cdots \gamma_t^{b_t}-1=e^{\Lambda}-1$. Then 
		$$
		\log |\Gamma|>-1.4\cdot 30^{t+3}\cdot t^{4.5} \cdot D^2 (1+\log D)(1+\log B)A_1\cdots A_t,
		$$
		where $B\ge \max\{|b_1|,\ldots,|b_t|\}$ and $A_i\ge \max\{Dh(\gamma_i),|\log \gamma_i|,0.16\}$ for $i=1,\ldots,t$.
	\end{theorem}
	
	In our application of Matveev's result (Theorem \ref{thm:Mat}), we need to ensure that the linear forms in logarithms are indeed nonzero. To ensure this, we shall need the following result given as Lemma 2.8 in \cite{GGL1}.

	\begin{lemma}[Lemma 2.8 in \cite{GGL1}]\label{lemGLm}
		Let \( N := N_{\mathbb{K}/\mathbb{Q}}, \) where \( \mathbb{K} = \mathbb{Q}(\alpha) \). Then
		\begin{enumerate}[(i)]
			\item For \( n, m \geq 1 \) and \( k \geq 2 \), \( |N(\alpha)| = 1 \).
			\item \( N(2\alpha - 1) = 2^{k+1} - 3 \) and \( N(f_k(\alpha)) = (k - 1)^2 / (2^{k+1}k^k - (k + 1)^{k+1}) \).
			\item For \( k \geq 2 \), \( N((2\alpha - 1)f_k(\alpha)) \le 1 \). This inequality is strict for $k\ge 3$ (and is an equality for $k=2$). 
		\end{enumerate}	
	\end{lemma}
	At some point, we treat cases with $t=2$, that is, linear forms in two logarithms. Let $A_1,~A_2>1$ be real numbers such that 
	\begin{equation}
		\label{eq:Aih}
		\log A_i\ge \max\left\{h(\gamma_i),\frac{|\log \gamma_i|}{D},\frac{1}{D}\right\}\quad {\text{\rm for}}\quad i=1,2.
	\end{equation}
	Put
	$$
	b':=\frac{|b_1|}{D\log A_2}+\frac{|b_2|}{D\log A_1}.
	$$
	The following  result is Corollary 2 in \cite{LMN}. 
	\begin{theorem}[Laurent et al., \cite{LMN}]
		\label{thm:LMNh}
		In case $t=2$, we put 
		$$\Lambda:=b_1\log \gamma_1-b_2\log \gamma_2, $$ where $|\gamma_{1}|, |\gamma_{2}|\ge 1$ are two multiplicatively independent real algebraic numbers and $b_1, b_2$ are positive integers. Then, we have
		$$
		\log |\Lambda|\ge -24.34 D^4\left(\max\left\{\log b'+0.14,\frac{21}{D},\frac{1}{2}\right\}\right)^2\log A_1\log A_2.
		$$
	\end{theorem}

	During the calculations, upper bounds on the variables are obtained which are too large, thus there is need to reduce them. To do so, we use some results from
	approximation lattices and the so--called LLL--reduction method from \cite{LLL}. We explain this in the following subsection.
	
	\subsection{Reduced Bases for Lattices and LLL--reduction methods}\label{sec2.3}
	Let $k$ be a positive integer. A subset $\mathcal{L}$ of the $k$--dimensional real vector space ${ \mathbb{R}^k}$ is called a lattice if there exists a basis $\{b_1, b_2, \ldots, b_k \}$ of $\mathbb{R}^k$ such that
	\begin{align*}
		\mathcal{L} = \sum_{i=1}^{k} \mathbb{Z} b_i = \left\{ \sum_{i=1}^{k} r_i b_i \mid r_i \in \mathbb{Z} \right\}.
	\end{align*}
	We say that $b_1, b_2, \ldots, b_k$ form a basis for $\mathcal{L}$, or that they span $\mathcal{L}$. We
	call $k$ the rank of $ \mathcal{L}$. The determinant $\text{det}(\mathcal{L})$, of $\mathcal{L}$ is defined by
	\begin{align*}
		\text{det}(\mathcal{L}) = | \det(b_1, b_2, \ldots, b_k) |,
	\end{align*}
	with the $b_i$'s being written as column vectors. This is a positive real number that does not depend on the choice of the basis (see \cite{Cas}, Section 1.2).
	
	Given linearly independent vectors $b_1, b_2, \ldots, b_k$ in $ \mathbb{R}^k$, we refer back to the Gram--Schmidt orthogonalization technique. This method allows us to inductively define vectors $b^*_i$ (with $1 \leq i \leq k$) and real coefficients $\mu_{i,j}$ (for $1 \leq j \leq i \leq k$). Specifically,
	\begin{align*}
		b^*_i &= b_i - \sum_{j=1}^{i-1} \mu_{i,j} b^*_j,~~~
		\mu_{i,j} = \dfrac{\langle b_i, b^*_j\rangle }{\langle b^*_j, b^*_j\rangle},
	\end{align*}
	where \( \langle \cdot , \cdot \rangle \)  denotes the ordinary inner product on \( \mathbb{R}^k \). Notice that \( b^*_i \) is the orthogonal projection of \( b_i \) on the orthogonal complement of the span of \( b_1, \ldots, b_{i-1} \), and that \( \mathbb{R}b_i \) is orthogonal to the span of \( b^*_1, \ldots, b^*_{i-1} \) for \( 1 \leq i \leq k \). It follows that \( b^*_1, b^*_2, \ldots, b^*_k \) is an orthogonal basis of \( \mathbb{R}^k \). 
	\begin{definition}
		The basis $b_1, b_2, \ldots, b_n$ for the lattice $\mathcal{L}$ is called reduced if
		\begin{align*}
			\| \mu_{i,j} \| &\leq \frac{1}{2}, \quad \text{for} \quad 1 \leq j < i \leq n,~~
			\text{and}\\
			\|b^*_{i}+\mu_{i,i-1} b^*_{i-1}\|^2 &\geq \frac{3}{4}\|b^*_{i-1}\|^2, \quad \text{for} \quad 1 < i \leq n,
		\end{align*}
		where $ \| \cdot \| $ denotes the ordinary Euclidean length. The constant $ {3}/{4}$ above is arbitrarily chosen, and may be replaced by any fixed real number $ y $ in the interval ${1}/{4} < y < 1$ (see \cite{LLL}, Section 1).
	\end{definition}
	Let $\mathcal{L}\subseteq\mathbb{R}^k$ be a $k-$dimensional lattice  with reduced basis $b_1,\ldots,b_k$ and denote by $B$ the matrix with columns $b_1,\ldots,b_k$. 
	We define
	\[
	l\left( \mathcal{L},y\right):= \left\{ \begin{array}{c}
		\min_{x\in \mathcal{L}}||x-y|| \quad  ;~~ y\not\in \mathcal{L}\\
		\min_{0\ne x\in \mathcal{L}}||x|| \quad  ;~~ y\in \mathcal{L}
	\end{array}
	\right.,
	\]
	where $||\cdot||$ denotes the Euclidean norm on $\mathbb{R}^k$. It is well known that, by applying the
	LLL--algorithm, it is possible to give in polynomial time a lower bound for $l\left( \mathcal{L},y\right)$, namely a positive constant $c_1$ such that $l\left(\mathcal{L},y\right)\ge c_1$ holds (see \cite{SMA}, Section V.4).
	\begin{lemma}\label{lem2.5}
		Let $y\in\mathbb{R}^k$ and $z=B^{-1}y$ with $z=(z_1,\ldots,z_k)^T$. Furthermore, 
		\begin{enumerate}[(i)]
			\item if $y\not \in \mathcal{L}$, let $i_0$ be the largest index such that $z_{i_0}\ne 0$ and put $\sigma:=\{z_{i_0}\}$, where $\{\cdot\}$ denotes the distance to the nearest integer.
			\item if $y\in \mathcal{L}$, put $\sigma:=1$.
		\end{enumerate}
		\noindent Finally, let 
		\[
		c_2:=\max\limits_{1\le j\le k}\left\{\dfrac{||b_1||^2}{||b_j^*||^2}\right\}.
		\]
		Then, 
		\[
		l\left( \mathcal{L},y\right)^2\ge c_2^{-1}\sigma^2||b_1||^2:=c_1^2.
		\]
	\end{lemma}
	In our application, we are given real numbers $\eta_0,\eta_1,\ldots,\eta_k$ which are linearly independent over $\mathbb{Q}$ and two positive constants $c_3$ and $c_4$ such that 
	\begin{align}\label{2.9}
		|\eta_0+a_1\eta_1+\cdots +a_k \eta_k|\le c_3 \exp(-c_4 H),
	\end{align}
	where the integers $a_i$ are bounded as $|a_i|\le A_i$ with $A_i$ given upper bounds for $1\le i\le k$. We write $A_0:=\max\limits_{1\le i\le k}\{A_i\}$. The basic idea in such a situation, from \cite{Weg}, is to approximate the linear form \eqref{2.9} by an approximation lattice. So, we consider the lattice $\mathcal{L}$ generated by the columns of the matrix
	$$ \mathcal{A}=\begin{pmatrix}
		1 & 0 &\ldots& 0 & 0 \\
		0 & 1 &\ldots& 0 & 0 \\
		\vdots & \vdots &\vdots& \vdots & \vdots \\
		0 & 0 &\ldots& 1 & 0 \\
		\lfloor C\eta_1\rfloor & \lfloor C\eta_2\rfloor&\ldots & \lfloor C\eta_{k-1}\rfloor& \lfloor C\eta_{k} \rfloor
	\end{pmatrix} ,$$
	where $C$ is a large constant usually of the size of about $A_0^k$ . Let us assume that we have an LLL--reduced basis $b_1,\ldots, b_k$ of $\mathcal{L}$ and that we have a lower bound $l\left(\mathcal{L},y\right)\ge c_1$ with $y:=(0,0,\ldots,-\lfloor C\eta_0\rfloor)$. Note that $ c_1$ can be computed by using the results of Lemma \ref{lem2.5}. Then, with these notations the following result  is Lemma VI.1 in \cite{SMA}.
	\begin{lemma}[Lemma VI.1 in \cite{SMA}]\label{lem2.6}
		Let $S:=\displaystyle\sum_{i=1}^{k-1}A_i^2$ and $T:=\dfrac{1+\sum_{i=1}^{k}A_i}{2}$. If $c_1^2\ge T^2+S$, then inequality \eqref{2.9} implies that we either have $a_1=a_2=\cdots=a_{k-1}=0$ and $a_k=-\dfrac{\lfloor C\eta_0 \rfloor}{\lfloor C\eta_k \rfloor}$, or
		\[
		H\le \dfrac{1}{c_4}\left(\log(Cc_3)-\log\left(\sqrt{c_1^2-S}-T\right)\right).
		\]
	\end{lemma}
	
	Finally, we present an analytic argument which is Lemma 7 in \cite{GL}.  
	\begin{lemma}[Lemma 7 in \cite{GL}]\label{Guz} If $ r \geq 1 $, $T > (4r^2)^r$ and $T >  \dfrac{p}{(\log p)^r}$, then $$p < 2^r T (\log T)^r.$$	
	\end{lemma}
	SageMath 9.5 is used to perform all computations in this work.

	\section{Proof of Theorem \ref{thm1.1}.}\label{Sec3}
	In this section, we prove Theorem \ref{thm1.1}. To do this, we first find some trivial solutions.
	\subsection{Trivial solutions}
	Here, we study equation \eqref{eq:main} in the following trivial scenarios.
	\begin{enumerate}[(a)]
		\item If $x=0$, then (assuming $0^0:=1$ when $n=0$ and  $k\ge 3$) \eqref{eq:main} becomes $L_m^{(k)}=1$, for which $m=1$. We therefore get $(n,m,k,x)=(n,1,k,0)$ with $n\ge 1$ for all $k\ge 2$.	
		\item If $x=1$, then the Diophantine equation \eqref{eq:main} becomes 
		\begin{align}\label{x1}
			L_{n+1}^{(k)}+L_{n}^{(k)}-L_{n-1}^{(k)}=L_m^{(k)}.
		\end{align}
		Moreover, the inequalities in \eqref{m_b} tells us that for $x=1$, we have $n-3 <m<n+6$, for all $k\ge 2$. This implies that 
		$   m\in\{n-2, n-1,n,n+1,\ldots,n+5\}$. A straightforward verification shows that the only solutions to \eqref{x1} with these options are $(n,m,k,x)=(0,2,k,1)$ with $k\ge 3$, 
		$(n,m,k,x)=(1,0,k,1)$ for all $k\ge 2$ and also the sporadic solution $(n,m,k,x)=(0,3,2,1)$.
	\end{enumerate}
	Assume from now on that $x\ge 2$.  Suppose next that $m\le 2$. Since $m\ge nx-2$, we get $nx\le 4$ and since $x\ge 2$, we get $n\in \{0,1,2\}$. Thus, $L_m^{(k)}\in \{2,1,3\}$, while 
	$$
	\left(L_{n+1}^{(k)}\right)^x+\left(L_n^{(k)}\right)^x-\left(L_{n-1}^{(k)}\right)^x\ge 1+2^x-1=2^x\ge 4>L_m^{(k)},
	$$
	so there are no solutions in this range. 
	
	For the remaining part of the proof, we assume $k\ge 2$, $x\ge 2$, $m\ge 3$. 
	
	\subsection{The case $k=2$ and $n=0,2$}
	
	Assume $k=2$. In the case $n=0$ we get the equation 
	$$
	1+2^x-(-1)^x=L_m.
	$$
	When $x$ is even, we get $L_m=2^x$ and it is known (for example, by Carmichael's primitive divisor theorem \cite{Car}) that the only powers of $2$ in the Lucas sequence are $m=0,~3$ for which $L_0=2,~L_3=2^2$ and the only one with even exponent $x$ is $L_3=2^2$.  We get the solution $(n,m,k,x)=(0,3,2,2)$. When $x$ is odd, we get $L_m=2^x+2$. When $x=1$, we get the solution $m=3$ which leads to $(n,m,k,x)=(0,3,2,1)$. When $x=2$, we get $L_m=6$, which is false since $6$ is not a member of the sequence of Lucas numbers. When $x\ge 3$, we get that $2\| L_m$ so $m$ is even (and multiple of $3$). In particular, $L_m=L_{m/2}^2\pm 2$. We thus get
	$$
	L_{m/2}^2\pm 2=2^x+2,\qquad {\text{\rm so}}\qquad L_{m/2}^2\in \{2^x,2^x+4\}.
	$$ 
	Since $x$ is odd, the equation $L_{m/2}^2=2^x$ is not possible. Thus, $L_{m/2}^2=2^x+4$ leading to $2\| L_{m/2}$ and $(L_{m/2}^2/2)^2=2^{x-2}+1$. This is a particular instance of the Catalan equation and its only solution is $x-2=3$ and $L_{m/2}/2=3$, leading to $L_{m/2}=6$, but this is wrong since $6$ is not a member of the sequence of Lucas numbers.  
	
	Assume $n=2$. In this case we get the equation
	$$
	4^x+3^x-1=L_m.
	$$
	The left--hand side is a multiple of $3$, so $m$ is even. Considering the values of $x$ modulo $4$, we get that the left--hand side is congruent to $1$ modulo $5$ (if $x\equiv 0,1\pmod 5$), to $4\pmod 5$ (if $x\equiv 2\pmod 4$) and to $0\pmod 5$ (if $x\equiv 3\pmod 4$). Since the Lucas sequence is periodic modulo $5$ with period $4$ and achieves the values 
	$2,1,3,4,2,1,3,4,\ldots$, we get that only the cases $L_m\equiv 1,4\pmod 5$ are possible and this implies $m\equiv 1,3\pmod 4$. This contradicts the fact that $m$ is even.

	\subsection{The case $n\le k$}\label{subsec3.2}
	\subsubsection{Bounding $x$ and $m$ in terms of $k$}
	We start by revisiting \eqref{eq2.2}. Assume $n\ge 3$. Then we rewrite \eqref{eq:main} as
	\begin{align*}
		L_m^{(k)}&=	\left(L_{n+1}^{(k)}\right)^x+\left(L_{n}^{(k)}\right)^x-\left(L_{n-1}^{(k)}\right)^x\\
		&=\left(3\cdot 2^{n-1}\right)^x+\left(3\cdot 2^{n-2}\right)^x-\left(3\cdot 2^{n-3}\right)^x.
	\end{align*}
	Assume $m\le k$. Then, since $m\ge 3$, we have $L_m^{(k)}=3\cdot 2^{m-2}$. If $n\ge 3$, then the equation above is
	\begin{align*}
		3\cdot2^{m-2}=\left(3\cdot 2^{n-1}\right)^x+\left(3\cdot 2^{n-2}\right)^x-\left(3\cdot 2^{n-3}\right)^x.
	\end{align*}
	The right--hand side above is divisible by $3^x$. So, $3^x$ must divide $3\cdot2^{m-2}$, which is false for $x\ge 2$.
	This was in case $n\ge 3$. We must consider the cases of small $n$ as well. If $n=0,1,2$, then since $k\ne 2$ when $n=0,2$, we have
	$$
	3\cdot 2^{m-2}=L_m^{(k)}=(L_{n+1}^{(k)})^x+(L_n^{(k)})^x-(L_{n-1}^{(k)})^x\in \{1+2^x,3^x+1-2^x,6^x+3^x-1\}
	$$
	(the first element above corresponds to $n=0$ and $k\ge 3$, the next element corresponds to $n=1$ and the third element corresponds to 
	$n=2$ and $k\ge 3$, respectively). We need to find out when the elements of the above set are of the form 
	$3\cdot 2^{m-2}$. The first is odd and the third is not a multiple of $3$, so they cannot be of the form  $3\cdot 2^{m-2}$. Finally, for the second element, we get the equation $3\cdot 2^{m-2}=3^x+1-2^x$. This has the solution $(x,m)=(2,3)$ and no other solution since 
	for $x\ge 3$, the right--hand side is at least $20$ and not a multiple of $8$, so it cannot be of the form $3\cdot 2^{m-2}$. Thus, we get the additional solution 
	$(n,m,k,x)=(1,3,k,2)$. Note that $k\ge 3$ since $3=m\le k$.  
	
	For the rest of this section, we assume $k\ge 2$, $x\ge 2$, $m\ge \max\{k+1,nx-2\}$. 
	
	To proceed, we prove the following result.
	\begin{lemma}\label{lem3.1}
		Let $n$, $m$, $k$, $x$ be integer solutions to Eq. \eqref{eq:main} with $k\ge 2$, $0\le n\le k$, $k\ge 2$ and $m\ge \max\{k+1,nx-2\}$, then	
		\begin{align*}
			x< 3.1\cdot 10^{14} k^5(\log k)^2 \log m.
		\end{align*}
	\end{lemma}
	\begin{proof}
		Assume first that $n\ge 3$. We go back to \eqref{eq:main} and rewrite it as
		\begin{align*}
			L_m^{(k)}&=\left(3\cdot 2^{n-1}\right)^x+\left(3\cdot 2^{n-2}\right)^x-\left(3\cdot 2^{n-3}\right)^x.
		\end{align*}	
		In case $n\in \{0,1,2\}$, then again using the fact that $k\ge 3$ when $n\in \{0,2\}$, the  right--hand side above is 
		\begin{equation}
			\label{eq:r}
			2^x+1,\quad 3^x+1-2^x, \quad 6^x+3^x-1.
		\end{equation}
		Now, using the Binet formula in \eqref{eq:Lnwitherror}, we have
		$$
		L_m^{(k)}-f_k(\alpha)(2\alpha-1)\alpha^{m-1}=e_k(m).
		$$
		Thus, assuming $n\ge 3$, we get
		\begin{align*}
			L_m^{(k)}-f_k(\alpha)(2\alpha-1)\alpha^{m-1}&=e_k(m)\\
			\left(3\cdot 2^{n-1}\right)^x+\left(3\cdot 2^{n-2}\right)^x-\left(3\cdot 2^{n-3}\right)^x-f_k(\alpha)(2\alpha-1)\alpha^{m-1}&=e_k(m)\\
			\left(3\cdot 2^{n-1}\right)^x-f_k(\alpha)(2\alpha-1)\alpha^{m-1}&=-\left(3\cdot 2^{n-2}\right)^x+\left(3\cdot 2^{n-3}\right)^x+e_k(m)\\
			1-3^{-x}2^{-(n-1)x}f_k(\alpha)(2\alpha-1)\alpha^{m-1}&=-\dfrac{1}{2^x}+\dfrac{1}{2^{2x}}+\dfrac{e_k(m)}{\left(3\cdot 2^{n-1}\right)^x}.
		\end{align*}	
		Taking absolute values and simplifying, we get
		\begin{align}
			\label{eq:gen}
			|\Gamma_1|:=	\left|3^{-x}2^{-(n-1)x}f_k(\alpha)(2\alpha-1)\alpha^{m-1}-1\right|<\dfrac{3}{2^x}.
		\end{align}
		This was for the case $n\ge 3$. For the cases when $n=0,1,2$, the amount $3\cdot 2^{n-1}$ above gets replaced by $2,3,6$, respectively and the right--hand sides get replaced by 
		$$
		\frac{1}{2^x}+\frac{|e_k(m)|}{2^x},\quad \left(\frac{2}{3}\right)^x+\frac{1}{3^x}+\frac{|e_k(m)|}{3^x},\quad \left(\frac{3}{6}\right)^x+\left(\frac{1}{6}\right)^x+\frac{|e_k(m)|}{6^x},
		$$
		respectively. All these expressions are bounded by $3(2/3)^x$. Thus, we get
		\begin{equation}
			\label{eq:gen1}
			|\Gamma_1|<3\left(\frac{2}{3}\right)^x,
		\end{equation}
		where now
		\begin{equation}
			\label{g1}
			\Gamma_1:=\delta^{-x} f_k(\alpha)(2\alpha-1)\alpha^{m-1}-1,\qquad \delta\in \{2,3,6\}.
		\end{equation}
		Note that $\Gamma_1\ne 0$ in all cases, otherwise we would have 
		$$
		f_k(\alpha)(2\alpha-1)\alpha^{m-1}\in \{(3\cdot 2^{n-1})^x,2^x,3^x,6^x\}.
		$$
		Applying norms in ${\mathbb K}={\mathbb Q}(\alpha)$ and using $|N(\alpha)|=1$, and item $(ii)$  and $(iii)$ of Lemma \ref{lemGLm}, the above equation becomes
		$$
		N\left(f_k(\alpha)\right)\cdot N(2\alpha-1)\in \{(3\cdot 2^{n-1})^{kx},2^{kx}, 3^{kx},6^{kx}\}.
		$$
		which implies that
		\begin{align*}
			1\ge \dfrac{(k - 1)^2 }{ 2^{k+1}k^k - (k + 1)^{k+1}}\cdot \left(2^{k+1}-3\right)=\left(2\right)^{kx}\ge \left(2\right)^{2\cdot 2}=16,
		\end{align*}
		a contradiction, for $x\ge 2$ and $k\ge 2$. Thus, $\Gamma_1\ne 0$.
		
		The algebraic number field containing the following $\gamma_i$'s is $\mathbb{K} := \mathbb{Q}(\alpha)$. When $n\ge 3$, we have $D = k$, $t :=4$,
		\begin{equation}\nonumber
			\begin{aligned}
				\gamma_{1}&:=(2\alpha-1)f_k(\alpha),\qquad \gamma_{2}:=\alpha, \qquad\gamma_{3}:=3,\qquad \gamma_{4}:=2,\\
				b_{1}&:=1,\qquad b_{2}:=m-1,\qquad b_{3}:=-x,\qquad b_{4}:=-(n-1)x.
			\end{aligned}
		\end{equation}
		We can take $A_3:=k \log 3$ and $A_4:=k \log 2$. Additionally, $h(\gamma_{2})=(\log \alpha)/k <0.7/k$, so we take $A_{2}:=0.7$. 
		For $A_1$, we first compute 
		\begin{align*}
			h(\gamma_{1}):=h\left((2\alpha-1)f_k(\alpha)\right)\le 
			h\left((2\alpha-1)\right)+h\left(f_k(\alpha)\right)<\log 3+3\log k<6\log k,
		\end{align*}
		for all $k\ge 2$. So, we can take $A_1:=6k\log k$. Next, $B \geq \max\{|b_i|:i=1,2,3, 4\}$, and by relation \eqref{m_b}, we can take $B:=m$. Now, by Theorem \ref{thm:Mat},
		\begin{align}\label{eq3.4}
			\log |\Gamma_1| &> -1.4\cdot 30^{7} \cdot 4^{4.5}\cdot k^2 (1+\log k)(1+\log m)\cdot 6k\log k \cdot (0.7) \cdot (k\log 3)\cdot (k\log 3)\nonumber\\
			&> -2.1\cdot 10^{14} k^5(\log k)^2 \log m.
		\end{align}
		For the cases $n=0,1,2$, we take $t:=3$ instead, keep the same $\gamma_1,~\gamma_2$ but $\gamma_3:= \delta$. Thus, we can take $A_3:=k\log 6$.  
		We then  get
		\begin{align}\label{eq3.41}
			\log |\Gamma_1| &> -1.4\cdot 30^{6} \cdot 4^{3.5}\cdot k^2 (1+\log k)(1+\log m)\cdot 6k\log k \cdot (0.7) \cdot (k\log 6))\nonumber\\
			&> -2.7\cdot 10^{12} k^4(\log k)^2 \log m.
		\end{align}
		When $n\ge 3$, comparing \eqref{eq:gen} and \eqref{eq3.4}, we get
		\begin{align*}
			x\log 2-\log 3&<2.1\cdot 10^{14} k^5(\log k)^2 \log m,
		\end{align*}
		which leads to $x<3.1\cdot 10^{14} k^5(\log k)^2 \log m$.	When $n=0,1,2$, applying instead \eqref{eq:gen1} we get 
		$$
		x\log(3/2)-\log 3<2.7\cdot 10^{12} k^4(\log k)^2 \log m,
		$$
		and this gives a smaller bound on $x$.
	\end{proof}

	Next, we prove the following.
	\begin{lemma}\label{lem3.2}
		Let $n$, $m$, $k$, $x$ be integer solutions to Eq. \eqref{eq:main} with $k\ge 2$, $0\le n\le k$, $k\ge 2$ and $m\ge \max\{k+1,5\}$, then	
		\begin{align*}
			m< 6.3\cdot 10^{32}k^{10}(\log k)^5 .
		\end{align*}
	\end{lemma}
	\begin{proof}
		If $n=0,1,2$, then 
		$$
		m< (n+3)x+3\le 5x+3<5\cdot 3.1\cdot 10^{14} k^5 (\log k)^2\log m+3<2\cdot 10^{15} k^5 (\log k)^2 \log m.
		$$
		Thus, 
		$$
		\frac{m}{\log m}<2\times 10^{15} k^5 (\log k)^2.
		$$
		Applying Lemma \ref{Guz} with $p:=m$, $r:=2$ and $T:=2\times 10^{15} k^5 (\log k)^2$ we get a better bound than the one from the statement of the lemma. 
		From now on, we assume $n\ge 3$. We go back to \eqref{eq:main} and rewrite it as
		\begin{align*}
			L_m^{(k)}&=\left(3\cdot 2^{n-1}\right)^x+\left(3\cdot 2^{n-2}\right)^x-\left(3\cdot 2^{n-3}\right)^x=\left(3\cdot 2^{n-2}\right)^x\left(2^x+1-2^{-x}\right).
		\end{align*}	
		As before, we use the Binet formula in \eqref{eq:Lnwitherror} and have
		\begin{align*}
			L_m^{(k)}-f_k(\alpha)(2\alpha-1)\alpha^{m-1}&=e_k(m)\\
			\left(3\cdot 2^{n-2}\right)^x\left(2^x+1-2^{-x}\right)-f_k(\alpha)(2\alpha-1)\alpha^{m-1}&=e_k(m)\\
			\left(3\cdot 2^{n-2}\right)^{x}\left(2^x+1-2^{-x}\right)\left(f_k(\alpha)\right)^{-1}(2\alpha-1)^{-1}\alpha^{-(m-1)}-1&=\dfrac{e_k(m)}{f_k(\alpha)(2\alpha-1)\alpha^{m-1}}.
		\end{align*}	
		Taking absolute values and simplifying, we get
		\begin{align}\label{g2}
			|\Gamma_2|&:=\left|\left(3\cdot 2^{n-2}\right)^{x}\left(2^x+1-2^{-x}\right)\left(f_k(\alpha)\right)^{-1}(2\alpha-1)^{-1}\alpha^{-(m-1)}-1\right|
			<\dfrac{1.5}{0.5\cdot1.5\alpha^{m-1}}
			=\dfrac{2}{\alpha^{m-1}}, 
		\end{align}
		where we have used the fact that $f_k(\alpha)>0.5$ and $2\alpha-1>1.5$, for all $k\ge 2$.
		Clearly, $\Gamma_2\ne 0$, otherwise we would have 
		\begin{align*}
			f_k(\alpha)= \dfrac{\left(3\cdot 2^{n-2}\right)^x\left(2^x+1-2^{-x}\right)}{(2\alpha-1)\alpha^{m-1}}.
		\end{align*}	
		Taking norms in ${\mathbb K}={\mathbb Q}(\alpha)$ as before, the above equation becomes
		\begin{align*}
			N\left(f_k(\alpha)\right)\cdot N(2\alpha-1)=\left(3\cdot 2^{n-2}\right)^{kx}\left(2^x+1-2^{-x}\right)^k.
		\end{align*}
		We have already shown before that the left--hand side above is $\le 1$. Therefore, 
		\begin{align*}
			1\ge 	N\left(f_k(\alpha)\right)\cdot N(2\alpha-1)=\left(3\cdot 2^{n-2}\right)^{kx}\left(2^x+1-2^{-x}\right)^k>500,
		\end{align*}
		a contradiction, for all $n\ge 3$, $x\ge 2$ and $k\ge 2$. Note that in the deduction above, we have used the fact that the function $\left(2^x+1-2^{-x}\right)$ is increasing and it is at least $ 4$ for all $x\ge 2$. Thus, $\Gamma_2\ne 0$. 
		
		The algebraic number field containing the following $\gamma_i$'s is $\mathbb{K} := \mathbb{Q}(\alpha)$, so $D = k$, $t :=4$,
		\begin{equation}\nonumber
			\begin{aligned}
				\gamma_{1}&:=3,\qquad \gamma_{2}:=2, \qquad\gamma_{3}:=\left(2^x+1-2^{-x}\right)/\left((2\alpha-1)f_k(\alpha)\right),\qquad \gamma_{4}:=\alpha,\\
				b_{1}&:=x\qquad b_{2}:=(n-2)x,\qquad b_{3}:=1,\qquad b_{4}:=-(m-1).
			\end{aligned}
		\end{equation}
		As before, $A_1:=k \log 3$, $A_2:=k \log 2$ and $A_{4}:=0.7$. To determine what $A_3$ could be, we first compute
		\begin{align*}
			h(\gamma_{3})&=h\left(\left(2^x+1-2^{-x}\right)/\left((2\alpha-1)f_k(\alpha)\right)\right)\le h\left(2^x+1-2^{-x}\right)+h(2\alpha-1)+h\left(f_k(\alpha)\right)\\
			&<2x\log 2+2\log 2+\log 3+3\log k\\
			&<2\left(3.1\cdot 10^{14} k^5(\log k)^2 \log m\right)\log 2+2\log 2+6\log k\\
			&<4.5\cdot 10^{14} k^5(\log k)^2 \log m,
		\end{align*}
		where we have used Lemma \ref{lem3.1}.
		Thus, we can take $A_3:=4.5\cdot 10^{14} k^6(\log k)^2 \log m$.
		
		Next, $B \geq \max\{|b_i|:i=1,2,3,4\}$, and by relation \eqref{m_b}, we can take $B:=m$. Now, by Theorem \ref{thm:Mat},
		\begin{align}\label{eq3.8}
			\log |\Gamma_2| &> -1.4\cdot 30^{7} \cdot 4^{4.5}\cdot k^2 (1+\log k)(1+\log m)\cdot k \log 3\cdot k \log 2\cdot  0.7\cdot 4.5\cdot 10^{14} k^6(\log k)^2 \log m\nonumber\\
			&> -6.8\cdot 10^{27} k^{10}(\log k)^3 (\log m)^2.
		\end{align}
		Comparing \eqref{g2} and \eqref{eq3.8}, we get
		\begin{align*}
			(m-1)\log\alpha-\log 2&<6.8\cdot 10^{27} k^{10}(\log k)^3 (\log m)^2,
		\end{align*}
		which leads to $m<1.5\cdot 10^{28} k^{10}(\log k)^3 (\log m)^2$. We now apply Lemma \ref{Guz} with  $p:=m$, $r:=2$, $T:=1.5\cdot 10^{28} k^{10}(\log k)^3 >(4r^2)^r=256$ for all $k\ge 2$. We get 
		\begin{align*}
			m&<2^2\cdot 1.5\cdot 10^{28} k^{10}(\log k)^3\left(\log 1.5\cdot 10^{28} k^{10}(\log k)^3\right)^2\\
			&=6\cdot 10^{28} k^{10}(\log k)^3 \left(\log (1.5\cdot 10^{28})+10\log k+3\log\log k\right)^2\\
			&<6\cdot 10^{28} k^{10}(\log k)^5  \left(\dfrac{65}{\log k}+10+\dfrac{3\log\log k}{\log k}\right)^2,
		\end{align*}
		so $m<6.3\cdot 10^{32}k^{10}(\log k)^5 $.	
	\end{proof}
	
	\subsubsection{The case $k>800$}\label{sub322}
	Here, we proceed by assuming for a moment that $k>800$. We get $m<6.3\cdot 10^{32}k^{10}(\log k)^5 <2^{0.28k}$, for all $k>800$. 
	Assume $n\ge 3$. We can rewrite \eqref{eq:main} as
	\begin{align*}
		L_m^{(k)}
		&=\left(3\cdot 2^{n-1}\right)^x+\left(3\cdot 2^{n-2}\right)^x-\left(3\cdot 2^{n-3}\right)^x.
	\end{align*}
	By part $(i)$ of Lemma \ref{fay} with $c:=0.28$, we have 
	\begin{align*}
		\left|\left(3\cdot 2^{n-1}\right)^x+\left(3\cdot 2^{n-2}\right)^x-\left(3\cdot 2^{n-3}\right)^x-3\cdot 2^{m-2}\right|<3\cdot 2^{m-2}\cdot \dfrac{4}{2^{0.72k}},
	\end{align*}
	which can be rewritten as 
	\begin{align}\label{eqc}
		\left|3^{x-1}(2^{2x}+2^x-1)-2^{m-2-(n-3)x}\right|<2^{m-2-(n-3)x} \cdot\frac{ 4}{2^{0.72 k}}.
	\end{align}
	Now, putting $y:=m-2-(n-3)x$ in \eqref{eqc} and dividing through  by $2^y$, we get
	\begin{align}\label{eq:0}
		\left|\frac{3^{x-1} (2^{2x}+2^x-1)}{2^y}-1\right|<\frac{4}{2^{0.72 k}}.
	\end{align}
	This was for $n\ge 3$. For $n=0,1,2$ (since $k\ge 3$ for $n=0,2$), the same argument leads to the following inequalities
	\begin{eqnarray*}
		\left|\frac{2^{x-(m-2)} (1+2^{-x})}{3}-1\right| & < & \frac{4}{2^{0.72k}},\qquad n=0;\\
		\left|\frac{3^{x-1}(1-(2^x-1)/3^x)}{2^{m-2}}-1\right| & < & \frac{4}{2^{0.72k}},\qquad n=1;\\
		\left|\frac{3^{x-1}(1+(3^x-1)/6^x)}{2^{m-2-x}}-1\right| & < & \frac{4}{2^{0.72k}}, \qquad n=2.
	\end{eqnarray*}
	The first one is false which shows that the case $n=0$ cannot hold for $k>800$. Indeed, either $x-(m-2)\le 1$ or $x-(m-2)\ge 2$. If $x-(m-2)\le 1$, then 
	$$
	\frac{2^{x-(m-2)} (1+2^{-x})}{3}\le \frac{2(1+1/2^2)}{3}=\frac{5}{6},\qquad {\text{\rm so}}\qquad \left|\frac{2^{x-(m-2)}(1+2^{-x})}{3}-1\right|\ge \frac{1}{6},
	$$
	while if $x-(m-2)\ge 2$, then 
	$$
	\frac{2^{x-(m-2)} (1+2^{-x})}{3}\ge \frac{2^2}{3}=\frac{4}{3},\qquad {\text{\rm so}}\qquad \left|\frac{2^{x-(m-2)}(1+2^{-x})}{3}-1\right|\ge \frac{1}{3}.
	$$
	Thus, in all cases, when $n=0$, we get $1/6<4/2^{0.72k}$, so $2^{0.72k}<24$, which is false for $k>800$. 
	
	Going back to the case $n\ge 3$, the left--hand side in \eqref{eq:0} is 
	\begin{align*}
		\left|e^{(x-1)\log 3+\log(2^{2x}+2^x+1)-y\log 2}-1\right|,
	\end{align*}
	and the right--hand side is $<1/2$. So, we get by a simple argument in calculus that
	\begin{align*}
		\left|(x-1)\log 3+\log(2^{2x}+2^x-1)-y\log 2\right|<\frac{8}{2^{0.72 k}}.
	\end{align*}
	This can be rewritten as 
	\begin{align*}
		\left|(x-1)\log 3+2x\log 2+\log\left(1+\frac{2^x-1}{2^{2x}}\right)-y\log 2\right|<\frac{8}{2^{0.72k}},
	\end{align*}
	or 
	\begin{equation}\label{eq:1}
		|(x-1)\log 3-z\log 2|<\frac{8}{2^{0.72k}}+\log\left(1+\frac{2^x-1}{2^{2x}}\right)<\frac{8}{2^{0.72k}}+\frac{2^x-1}{2^{2x}},
	\end{equation}
	with $z:=y-2x$. Since $x\ge 2$,  the right--hand side in \eqref{eq:1} is $<6/16$, so $z>0$.  The cases $n=1,~2$ lead, by similar arguments, to 
	\begin{equation}\label{eq:1prime}
		|(x-1)\log 3-z\log 2|<\frac{8}{2^{0.72k}}+\delta(x),\qquad \delta(x):=\begin{cases} 
			-\log\left(1-\dfrac{2^x-1}{3^x}\right), & \text{if} \ n=1, \\
			\log\left(1+\dfrac{3^x-1}{6^x}\right), & \text{if} \ n=2,\end{cases}
	\end{equation}
	and $z:=m-2$ when $n=1$ and $z:=m-2-x$, when $n=2$. In both cases above the right--hand sides are $<3/4$ so $z>0$. We get that 
	$$
	\delta(x)\le \frac{2(2^x-1)}{3^x}\quad (n=1)\qquad {\text{\rm and}}\qquad \delta(x)<\frac{3^x-1}{6^x}\qquad (n=2).
	$$
	Further, 
	$$
	z\log 2\le (x-1)\log 3+|(x-1)\log 3-z\log 2|<(x-1)\log 3+\frac{3}{4},
	$$
	so 
	\begin{equation}
		\label{eq:2}
		\frac{z}{\log 3}<\frac{x-1}{\log 2} +\frac{3}{4(\log 2\log 3)}<\frac{x-1}{\log 2}+1<2x.
	\end{equation}
	The left--hand side in \eqref{eq:1} (or \eqref{eq:1prime}) is of the form 
	\begin{align*}
		|b_1\log \gamma_1-b_2\log \gamma_2|,
	\end{align*}
	where $\gamma_1:=3,~\gamma_2:=2,~b_1:=x-1$ and $b_2:=z$. The numbers $\gamma_1$, $\gamma_2$ are rational so $D=1$. We take 
	\begin{align*}
		\log A_1\ge \max\left\{h(\gamma_{1}), |\log \gamma_1|, 1\right\},\quad \log A_2\ge \max\left\{h(\gamma_{2}), |\log \gamma_2|, 1\right\},
	\end{align*}
	so that $\log A_1=\log 3$ and $\log A_2=1$. 
	Further,
	\begin{align*}
		b'=\frac{b_1}{D\log A_2}+\frac{b_2}{D\log A_1}=x-1+\frac{z}{\log 3}<x-1+2x<3x,
	\end{align*}
	by \eqref{eq:2}. Since $\gamma_1,~\gamma_2$ are real, positive and multiplicatively independent, Theorem \ref{thm:LMNh} shows that 
	\begin{align}\label{eq:b}
		|(x-1)\log 3-z\log 2|>\exp\left(-24.34 \left(\max\left\{21, \log b'+0.14\right\}\right)^2\log 3\right).
	\end{align}
	Since 
	\begin{align*}
		\frac{8}{2^{0.72k}}+\delta(x)\le \frac{8}{2^{0.72k}}+\frac{2}{2^{x/2}}\le \frac{10}{2^{\min\{0.72k,rx\}}},
	\end{align*}
	where 
	\begin{equation}
		\label{eq:delta}
		\delta(x)\in\left\{\frac{2^x-1}{2^{2x}},\frac{2(2^x-1)}{3^x},\frac{3^x-1}{6^x}\right\},
	\end{equation}
	and we take $r=1$ except if $n=1$ when $r=1/2$, we compare \eqref{eq:1} and \eqref{eq:b} to conclude that
	\begin{equation}
		\label{eq:3}
		\min \{0.72 k,rx\} \log 2-\log 10<24.34 (\max\{21,\log(3x)+0.14\})^2\log 3.
	\end{equation}
	We proceed in two cases, when $\min \{0.72 k,x\}$ is $0.72k$ or $rx$.
	
	\medskip
	
	\begin{enumerate}[(i)]
		\item Assume first that $\min \{0.72 k,rx\}=0.72k$, then
		$$
		(0.72 \log 2)k<24.34( \max\{21, \log(3x)+0.14\})^2\log 3+\log 10.
		$$
		If $\max\{21, \log(3x)+0.14\}=21$, we get 
		$
		(0.72\log 2)k<24.34\cdot (21)^2\log 3+\log 10$, so $k<24000$.
		If $\max\{21, \log(3x)+0.14\}= \log(3x)+0.14$, we get 
		$$
		(0.72\log 2)k <24.34\log 3 (\log(3x)+0.14)^2+\log 10,
		$$
		so that
		$$
		k<54 (\log(3x)+0.14)^2+5<59(\log(3x)+0.14))^2,
		$$
		for all $x\ge 2$. By Lemma \ref{lem3.2}, we have 
		$$
		\log m<\log 6.3+32\log 10+10(\log k)+5\log\log k,
		$$
		and by Lemma \ref{lem3.1},
		$$
		x<3.1\cdot 10^{14} k^5 (\log k)^2 \log m.
		$$
		Thus, 
		\begin{align*}
			x  &<  3.1\cdot 10^{14} (59(\log(3x)+0.14)^2)^{5} (\log 59+2\log(\log(3x)+0.14))^2\\
			&\times   (\log 6.3+32\log 10+10(\log(59(\log(3x)+0.14)^2))+5\log(\log 59+2\log(\log(3x)+0.14))).
		\end{align*}
		We get $x<10^{47}$.
		
		\medskip 
		
		\item Next, if $\min \{0.72 k,rx\}=rx$, we then get that
		$$
		rx\log 2-\log 10<24.34(\max\{21,\log(3x)+0.10\})^2\log 3.
		$$
		If the maximum in the right above is in $21$, we then get $rx<18000$, while if the maximum is in $\log(3x)+0.14$, we then get
		$$
		rx\log 2<24.34\log 3 (\log(3x)+0.14)^2+\log 10\le 24.34\log 3(\log(3rx)+0.14+\log 2)^2+\log 10,
		$$
		so that $rx<4000$.
	\end{enumerate} 
	
	\medskip 
	
	So, in all cases $x<10^{47}$. We now return to \eqref{eq:1} which we write as 
	\begin{align*}
		\left|\frac{\log 3}{\log 2}-\frac{z}{x-1}\right|<\frac{8/2^{0.72k}+\delta(x)}{(x-1)\log 2}<\frac{12/2^{0.72k}+2\delta(x)}{x-1},
	\end{align*}
	where $\delta(x)$ is one of the three functions appearing in \eqref{eq:delta} (and $\delta(x)=(2^x-1)/2^{2x}$ if $n\ge 3$).  
	We generated the $97$th convergent $p_{97}/q_{97}=[a_0,a_1,\ldots,a_{97}]=[1,1,\ldots,3]$ of $\log 3/\log 2$ and we got that $q_{97}>10^{47}$ and that  $\max\{a_k: 0\le k\le 97\}=55$. By well--known properties of continued fractions, we get 
	\begin{align*}
		\frac{1}{57(x-1)^2}<\left|\frac{\log 3}{\log 2}-\frac{y-2x}{x-1}\right|<\frac{12/2^{0.72k}+2\delta(x)}{x-1}.
	\end{align*}
	Therefore, we have
	\begin{align*}
		\frac{1}{57(x-1)}<\frac{12}{2^{0.72k}}+2\delta(x).
	\end{align*}
	The left--hand side is at least $1/(57\cdot 10^{47})$, while $12/2^{0.72k}<10^{-170}$ since $k>800$. So, we get
	$$
	\frac{1}{58 (x-1)}<2\delta(x),
	$$
	which gives $x\le 10$ except if $n=1$ for which $x\le 20$. The cases $n=1,2$ lead to $m<(n+3)x+3<100<k$, a contradiction. Thus, $x\le 10$, $n\ge 3$, $z=y-2x$ and  
	\begin{equation}
		\label{eq:4}
		|(x-1)\log 3-(y-2x)\log 2|<\frac{8}{2^{0.72k}}+\frac{2^x-1}{2^{2x}}.
	\end{equation}
	We used a simple code to compute the minimum values of the expression 
	\[
	|(x-1)\log 3 - (y-2x)\log 2|,
	\]
	for integer values of \( y - 2x \). By assigning values to \( x \) ranging from 2 to 10, the code found the \( y \) that minimized the expression for each \( x \). The results showed that the left--hand side of \eqref{eq:4} was at least 
	\[
	0.28,~0.11,~0.16,~0.23,~0.05,~0.33,~0.06,~0.22,~0.18,
	\]
	for \( x = 2, 3, \ldots, 10 \), respectively. It turns out that for no $x\in \{2,3,\ldots,10\}$ inequality \eqref{eq:4} is satisfied. This finishes the argument that the case $k>800$ is not possible. 
	
	\subsubsection{The case $k\le800$}
	To proceed, we go back to \eqref{eq:gen} with the assumption that $3\le n \le k$ and write
	\begin{align*}
		|\Lambda_1|:=|\log(\Gamma_1+1)|=\left|\log f_k(\alpha)+\log(2\alpha-1)+(m-1)\log\alpha-x\log 3-(n-1)x\log 2\right|<\dfrac{4.5}{2^x},
	\end{align*}
	where $a_1 := 1$, $a_2 := 1$, $a_3 := m-1$, $a_4 := -x$, $a_5 := -(n-1)x$ are integers with 
	\[
	\max\{|a_i| : 1 \leq i \leq 5\}  < m<6.3 \cdot 10^{32} k^{10} (\log k)^5< 10^{66},
	\]
	where we used Lemma \ref{lem3.2}, for all $k\le 800$. For each $k \in [2, 800]$, we used the LLL--algorithm to compute a lower bound for the smallest nonzero number of the form $|\Lambda_1|$, with integer coefficients $a_i$ not exceeding $6.3 \cdot 10^{32} k^{10} (\log k)^5$ in absolute value. Specifically, we consider the approximation lattice
	$$ \mathcal{A}=\begin{pmatrix}
		1 & 0 & 0 & 0 & 0 \\
		0 & 1 & 0 & 0 & 0 \\
		0 & 0 & 1 & 0 & 0 \\
		0 & 0 & 0 & 1 & 0 \\
		\lfloor  C\log f_k(\alpha)\rfloor & \lfloor C\log (2\alpha-1) \rfloor& \lfloor C\log \alpha \rfloor& \lfloor C\log 3 \rfloor & \lfloor C\log2 \rfloor 
	\end{pmatrix} ,$$
	with $C:= 10^{331}$ and choose $y:=\left(0,0,0,0,0\right)$. Now, by Lemma \ref{lem2.5}, we get $$c_2=10^{-69}\qquad\text{and}\qquad l\left(\mathcal{L},y\right)\ge c_1:=3.13\cdot10^{68}.$$
	So, Lemma \ref{lem2.6} gives $S=5\cdot 10^{132}$ and $T=2.5\cdot 10^{66}$. Since $c_1^2\ge T^2+S$, then choosing $c_3:=4.5$ and $c_4:=\log2$, we get $x\le 872$.
	
	Next, if $n\in\{0,1,2\}$, we instead go back to \eqref{g1} together with \eqref{eq:gen1} and write
	\begin{align*}
		\left|\log f_k(\alpha)+\log(2\alpha-1)+(m-1)\log\alpha-x\log \delta\right|<\dfrac{4.5}{1.5^x},
	\end{align*}
	where $a_1 := 1$, $a_2 := 1$, $a_3 := m-1$, $a_4 := -x$ are integers with $\max\{|a_i| : 1 \leq i \leq 5\}  < 10^{66}$,
	as before. Again, for each $k \in [2, 800]$ and each $\delta\in\{2,3,6\}$, we used the LLL--algorithm to compute a lower bound for the smallest nonzero number of the form above and consider the approximation lattice
	$$ \mathcal{A}=\begin{pmatrix}
		1 & 0 & 0 & 0  \\
		0 & 1 & 0 & 0  \\
		0 & 0 & 1 & 0  \\
		\lfloor  C\log f_k(\alpha)\rfloor & \lfloor C\log (2\alpha-1) \rfloor& \lfloor C\log \alpha \rfloor& \lfloor C\log \delta \rfloor  
	\end{pmatrix} ,$$
	with $C:= 10^{265}$ and choose $y:=\left(0,0,0,0\right)$. So, Lemma \ref{lem2.5} gives $c_2=10^{-70}$ and $l\left(\mathcal{L},y\right)\ge c_1:=2.94\cdot10^{68}$.
	Moreover, Lemma \ref{lem2.6} gives $S=4\cdot 10^{132}$ and $T=2\cdot 10^{66}$. Since $c_1^2\ge T^2+S$, then choosing $c_3:=4.5$ and $c_4:=\log1.5$, we get $x\le 1119$. 
	
	Thus, $x\le 1119$ for all cases when $0\le n\le k\le 800$.
	
	Finally in this subsection,  we go back to \eqref{g2} and write $|\Lambda_2|:=|\log(\Gamma_2+1)|$ as
	\begin{align*}
		\left|\log f_k(\alpha)+\log(2\alpha-1)+(m-1)\log\alpha-x\log 3-(n-2)x\log 2-\log \left(2^x+1-2^{-x}\right)\right|<\dfrac{3}{\alpha^{m-1}}
	\end{align*}
	where $a_1 := 1$, $a_2 := 1$, $a_3 := m-1$, $a_4 := -x$, $a_5 := -(n-2)x$ and $a_6:=-1$ (in case $n\ge 3$) are integers with 
	\[
	\max\{|a_i| : 1 \leq i \leq 5\}  < m<6.3 \cdot 10^{32} k^{10} (\log k)^5< 10^{66},
	\]
	where we used Lemma \ref{lem3.2}, for all $k\le 800$. So, for each $k \in [2, 800]$ and each $x\in [1,1119]$ we used the LLL--algorithm to compute a lower bound for the smallest nonzero number of the form $|\Lambda_2|$, with integer coefficients $a_i$ not exceeding $6.3 \cdot 10^{32} k^{10} (\log k)^5$ in absolute value. That is, we instead consider the approximation lattice
	$$ \mathcal{A}=\begin{pmatrix}
		1 & 0 & 0 & 0 & 0 & 0 \\
		0 & 1 & 0 & 0 & 0 & 0\\
		0 & 0 & 1 & 0 & 0 & 0\\
		0 & 0 & 0 & 1 & 0 & 0\\
		0 & 0 & 0 & 0 & 1 & 0\\
		\lfloor  C\log f_k(\alpha)\rfloor & \lfloor C\log (2\alpha-1) \rfloor& \lfloor C\log \alpha \rfloor& \lfloor C\log 3 \rfloor & \lfloor C\log2 \rfloor & \lfloor C\log\left(2^x+1-2^{-x}\right) \rfloor 
	\end{pmatrix} ,$$
	with $C:= 10^{396}$ and choose $y:=\left(0,0,0,0,0,0\right)$. Now, by Lemma \ref{lem2.5}, we get $$c_2:=10^{-69}\qquad\text{and}\qquad l\left(\mathcal{L},y\right)\ge c_1:=3.87\cdot10^{68}.$$
	So, Lemma \ref{lem2.6} gives $S=6\cdot 10^{132}$ and $T=3\cdot 10^{66}$. Since $c_1^2\ge T^2+S$, then choosing $c_3:=3$ and $c_4:=\log\alpha$, we get $m\le 1474$.
	This was for $n\ge 3$. If $n\in \{0,1,2\}$, then $m<(n+3)x+3\le 5x+3\le 5598$. 
	
	Therefore, for the case when $0\le n\le k$, we have that all solutions $n$, $m$, $k$, $x$ to \eqref{eq:main} satisfy $k\in [2,800]$, $m\in [3,5597]$, $n\in [0,\min\{k,\lfloor(m+3)/x\rfloor\}]$ and $x\in[1,1119]$. A computational search in SageMath within these specified ranges for solutions to the Diophantine equation \eqref{eq:main} yielded only the solutions presented in the Theorem \ref{thm1.1}. To efficiently handle large \( L_n^{(k)} \) values, our SageMath 9.5 code utilized batch processing, iterating through all \( (n, m,k, x) \) combinations within these ranges. 
	
	\subsection{The case $n> k$}
	\subsubsection{Bounding $x$, $n$ and $m$ in terms of $k$}
	Again, we proceed as in Subsection \ref{subsec3.2}. We start by proving a series of results.
	\begin{lemma}\label{lem3.3}
		Let $n$, $m$, $k$, $x$ be integer solutions to Eq. \eqref{eq:main} with $n>k\ge 2$, $x\ge 2$ and $m\ge 3$, then	
		\begin{align*}
			x<2.6\cdot 10^{15} n k^4(\log k)^3\log n .
		\end{align*}	
	\end{lemma}
	\begin{proof}
		We go back to \eqref{eq:main} and rewrite it using the Binet formula in \eqref{eq:Lnwitherror} as
		\begin{align*}
			L_m^{(k)}-f_k(\alpha)(2\alpha-1)\alpha^{m-1}&=e_k(m)\\
			f_k(\alpha)(2\alpha-1)\alpha^{m-1}-\left(L_{n+1}^{(k)}\right)^{x}&=\left(L_{n}^{(k)}\right)^x-\left(L_{n-1}^{(k)}\right)^x-e_k(m)\\
			0<f_k(\alpha)(2\alpha-1)\alpha^{m-1}\left(L_{n+1}^{(k)}\right)^{-x}-1&=\left(\dfrac{L_{n}^{(k)}}{L_{n+1}^{(k)}}\right)^x-\left(\dfrac{L_{n-1}^{(k)}}{L_{n+1}^{(k)}}\right)^x-\dfrac{e_k(m)}{\left(L_{n+1}^{(k)}\right)^x}\\
			&<\left(\dfrac{L_{n}^{(k)}}{L_{n+1}^{(k)}}\right)^x<0.7^x<\dfrac{1}{1.4^x},
		\end{align*}	
		The above calculation is justified for $n>k\ge 2$ and $x\ge 2$, because then $n\ge 3$ so $L_n^{(k)}>L_{n-1}^{(k)}> 1$, so since $x\ge 2$, 
		$$
		\left(L_{n}^{(k)}\right)^x\ge \left(L_{n-1}^{(k)}+1\right)^x>\left(L_{n-1}^{(k)}\right)^x+x\ge \left(L_{n-1}^{(k)}\right)^x+2>\left(L_{n-1}^{(k)}\right)^x+e_k(m)
		$$
		and 
		$$
		\left(L_{n-1}^{(k)}\right)^x\ge 2^x\ge 4>1.5>-e_{k}(m),\qquad {\text{\rm so}}\qquad (L_{n-1}^{(k)})^x+e_k(m)>0.
		$$
		We thus get
		\begin{align}\label{g3}
			|\Gamma_3|:=\left|f_k(\alpha)(2\alpha-1)\alpha^{m-1}\left(L_{n+1}^{(k)}\right)^{-x}-1\right|<\dfrac{1}{1.4^x}.
		\end{align}
		Note that $\Gamma_3\ne 0$, otherwise we would have 
		\begin{align*}
			f_k(\alpha)= \dfrac{\left(L_{n+1}^{(k)}\right)^x}{(2\alpha-1)\alpha^{m-1}}.
		\end{align*}	
		Taking norms in ${\mathbb K}={\mathbb Q}(\alpha)$ as before, the above equation becomes
		\begin{align*}
			N\left(f_k(\alpha)\right)\cdot N(2\alpha-1)=\left(L_{n+1}^{(k)}\right)^{kx}.
		\end{align*}
		We have already shown before that the left--hand side above is $\le 1$. Therefore, 
		\begin{align*}
			1\ge 	N\left(f_k(\alpha)\right)\cdot N(2\alpha-1)=\left(L_{n+1}^{(k)}\right)^{kx}\ge \left(L_{4}^{(2)}\right)^4=2401,
		\end{align*}
		a contradiction, for all $n>k\ge 2$ and $x\ge 2$. 
		
		The algebraic number field containing the following $\gamma_i$'s is $\mathbb{K} := \mathbb{Q}(\alpha)$. We have $D = k$, $t :=3$,
		\begin{equation}\nonumber
			\begin{aligned}
				\gamma_{1}&:=L_{n+1}^{(k)}, \qquad\gamma_{2}:=(2\alpha-1)f_k(\alpha),\qquad \gamma_{3}:=\alpha,\\
				b_{1}&:=-x,\qquad  b_{2}:=1,\qquad b_{3}:=m-1.
			\end{aligned}
		\end{equation}
		Since $h\left(L_{n+1}^{(k)}\right)\le h(2\alpha^n)<n\log2\alpha<n\log 4<1.4n$, we can take $A_1:=1.4kn$. As before, we take $A_{2}:=6k\log k$, $A_{3}:=0.7$ and  $B:=m$. Now, by Theorem \ref{thm:Mat},
		\begin{align}\label{eq3.12}
			\log |\Gamma_3| &> -1.4\cdot 30^{6} \cdot 3^{4.5}\cdot k^2 (1+\log k)(1+\log m)\cdot 1.4kn\cdot   0.7\cdot 6k\log k\nonumber\\
			&> -4.0\cdot 10^{12} n k^4(\log k)^2 \log m.
		\end{align}
		Comparing \eqref{g3} and \eqref{eq3.12}, we get
		\begin{align*}
			x\log 1.4&<4.0\cdot 10^{12} n k^4(\log k)^2 \log m,\\
			x&<1.2\cdot 10^{13} n k^4(\log k)^2 \log m.
		\end{align*}
		Now, by inequality \eqref{m_b}, $m<(n+3)x+3$, hence 
		\begin{align*}
			x&<1.2\cdot 10^{13} n k^4(\log k)^2 \log ((n+3)x+3)\\
			&<2.7\cdot 10^{13} n k^4(\log k)^2 \log (nx).
		\end{align*}
		Now, multiplying $n$ on both sides of the above inequality, we get 
		$nx<2.7\cdot 10^{13} n^2 k^4(\log k)^2 \log (nx)$. We now apply Lemma \ref{Guz} with  $p:=nx$, $r:=1$, $T:=2.7\cdot 10^{13} n^2 k^4(\log k)^2$ and have 
		\begin{align*}
			nx&<2\cdot 2.7\cdot 10^{13} n^2 k^4(\log k)^2\log 2.7\cdot 10^{13} n^2 k^4(\log k)^2\\
			&=5.4\cdot 10^{13} n^2 k^4(\log k)^2 \left(\log (2.7\cdot 10^{13})+2\log n+4\log k+2\log\log k\right)\\
			&<5.4\cdot 10^{13} n^2 k^4(\log k)^3\log n  \left(\dfrac{31}{\log n\log k}+\dfrac{2}{\log k}+\dfrac{4}{\log n}+\dfrac{2\log\log k}{\log n\log k}\right),
		\end{align*}
		so that $nx<2.6\cdot 10^{15} n^2 k^4(\log k)^3\log n $. Dividing both sides by $n$, we get $x<2.6\cdot 10^{15} n k^4(\log k)^3\log n $.
	\end{proof}
	
	To proceed, we now work under the assumption that $n > 800$, so that we can explicitly determine an upper bound for $n$, $m$ and $x$ in terms of $k$ only. We prove the following result.
	\begin{lemma}\label{lem3.4}
		Let $n$, $m$, $k$, $x$ be integer solutions to Eq. \eqref{eq:main} with $n>\max\{800,k\}$, $x\ge 2$ and $m\ge 3$, then	
		\begin{align*}
			n<5.5\cdot 10^{27} k^6(\log k)^6, \qquad  x< 8.7\cdot 10^{30}k^7(\log k)^8\qquad\text{and}\qquad m<5.6\cdot 10^{44}k^{10}(\log k)^{12}.	\end{align*}	
	\end{lemma}
	\begin{proof}
		Since $n>k$, then Lemma \ref{lem3.3} tells us that 
		\begin{align}\label{x_bound1}
			x<2.6\cdot 10^{15} n^5(\log n)^4.
		\end{align}
		So, for each $i\in\{-1,0,1\}$, the inequality 
		\begin{align*}
			\dfrac{x}{\alpha^{n+i-1}}<\dfrac{2.6\cdot 10^{15} n^5(\log n)^4}{\alpha^{n-2}}<\dfrac{1}{\alpha^{0.7n}},
		\end{align*}
		holds for all $n>800$. Now, Lemma \ref{lem:Lnx} tells us that 
		\begin{align}\label{eq3.14}
			\left(L_{n+i}^{(k)}\right)^x = f_k(\alpha)^x (2\alpha-1)^x\alpha^{(n+i-1)x}(1 + \eta_{n}),\qquad\text{with}\qquad	
			|\eta_{n}| < \dfrac{1.5xe^{1.5x/\alpha^{n+i-1}}}{\alpha^{n+i-1}}<\dfrac{10}{\alpha^{0.7n}}.
		\end{align}
		So, we rewrite \eqref{eq:main} using \eqref{eq3.14} as
		\begin{align*}
			f_k(\alpha)(2\alpha-1)\alpha^{m-1}&-f_k(\alpha)^x(2\alpha-1)^x\alpha^{(n-1)x}\left(1+\alpha^{-x}-\alpha^{-2x}\right)\\
			&=f_k(\alpha)^x(2\alpha-1)^x\alpha^{(n-1)x}\left(1+\alpha^{-x}-\alpha^{-2x}\right)\eta_n+e_k(m).
		\end{align*}
		Dividing both sides of the above equality by $f_k(\alpha)^x(2\alpha-1)^x\alpha^{(n-1)x}$ and taking absolute values, we get
		\begin{align*}
			&\left|f_k(\alpha)(2\alpha-1)\alpha^{m-1}f_k(\alpha)^{-x}(2\alpha-1)^{-x}\alpha^{-(n-1)x}-\left(1+\alpha^{-x}-\alpha^{-2x}\right)\right|\\
			&<|\eta_n|\left(1+\alpha^{-x}-\alpha^{-2x}\right)+\dfrac{1.5}{f_k(\alpha)^x(2\alpha-1)^x\alpha^{(n-1)x}}\\
			&<\dfrac{15}{\alpha^{0.7n}}+\dfrac{1.5}{\alpha^{(n-1)x}}<\dfrac{18}{\alpha^{0.7n}},
		\end{align*}
		where we have used the fact that $0<1+\alpha^{-x}-\alpha^{-2x}<1.5$, 
		$f_k(\alpha)(2\alpha-1)\alpha^{n-1}>\alpha^{n-1}$ and $(n-1)x>0.7n$, for all $n>800$, $x\ge 2$ and $k\ge 2$. This means that
		\begin{align}\label{g4}
			|\Gamma_4|:=	\left|f_k(\alpha)^{1-x}(2\alpha-1)^{1-x}\alpha^{m-1-(n-1)x}-1\right|<\dfrac{18}{\alpha^{0.7n}}+\dfrac{1}{\alpha^x}+\dfrac{1}{\alpha^{2x}}<\dfrac{27}{\alpha^{\min\{0.7n,x\}}}.
		\end{align}
		Note that $\Gamma_4\ne 0$, otherwise we would have 
		\begin{align}\label{fk}
			\left((2\alpha-1)f_k(\alpha)\right)^{x-1}=\alpha^{m-1-(n-1)x}.
		\end{align}	
		Applying norms in ${\mathbb K}={\mathbb Q}(\alpha)$ and using $|N(\alpha)|=1$, equation \eqref{fk} becomes
		\begin{align*}
			|N\left((2\alpha-1)f_k(\alpha)\right)|^{x-1}=1,
		\end{align*}
		which implies that  $x=1$, for $k\ge 3$, contradicting the working assumption that $x\ge 2$. For the special case when $k=2$, we have that $(2\alpha-1)f_k(\alpha)=\alpha$, so that \eqref{fk} becomes $\alpha^{x-1}=\alpha^{m-1-(n-1)x}$, giving $m=nx$. At this point, we go back and rewrite \eqref{eq:main} with $k=2$ and $m=nx$ as
		\begin{align*}
			\left(L_{n+1}\right)^x+\left(L_{n}\right)^x-\left(L_{n-1}\right)^x=L_{nx}.
		\end{align*}
		Now, since $\left(L_{n}\right)^x-\left(L_{n-1}\right)^x>0$ for all $x\ge 2$ and $n\ge 4$, we have by the Binet formula $L_n=\alpha^n+\beta^n$, that
		\begin{align*}
			\left(\alpha^{n+1}-1\right)^x<\left(\alpha^{n+1}+\beta^{n+1}\right)^x= \left(L_{n+1}\right)^x\le \left(L_{n+1}\right)^x+\left(L_{n}\right)^x-\left(L_{n-1}\right)^x=L_{nx}< \alpha^{nx}+1.
		\end{align*}
		This leads to
		\begin{align*}
			\alpha^x\left(1-\frac{1}{\alpha^{n+1}}\right)^x< 1+\frac{1}{\alpha^{nx}},
		\end{align*}
		so that 
		\begin{align*}
			2<	\alpha^2\left(1-\frac{1}{\alpha^{4+1}}\right)^2\le \alpha^x\left(1-\frac{1}{\alpha^{n+1}}\right)^x< 1+\frac{1}{\alpha^{nx}}\le  1+\frac{1}{\alpha^{4\cdot 2}}<1.1,
		\end{align*}
		which is also a contradiction. Thus, $\Gamma_4\ne 0$.
		
		The algebraic number field containing the following $\gamma_i$'s is $\mathbb{K} := \mathbb{Q}(\alpha)$. We have $D = k$, $t :=2$,
		\begin{equation}\nonumber
			\begin{aligned}
				\gamma_{1}&:=(2\alpha-1)f_k(\alpha),\qquad \gamma_{2}:=\alpha,\\
				b_{1}&:=1-x,\qquad b_{2}:=m-1-(n-1)x.
			\end{aligned}
		\end{equation}
		As before, we take $A_{1}:=6k\log k$ and $A_{2}:=0.7$. By the second inequality in \eqref{m_b}, we can take $B:=x$. So, Theorem \ref{thm:Mat} gives
		\begin{align}\label{eq3.16}
			\log |\Gamma_4| &> -1.4\cdot 30^{5} \cdot 2^{4.5}\cdot k^2 (1+\log k)(1+\log x)\cdot  0.7\cdot 6k\log k\nonumber\\
			&> -2.0\cdot 10^{10} k^3(\log k)^2 \log x.
		\end{align}
		Comparing \eqref{g4} and \eqref{eq3.16}, we get
		\begin{align*}
			\min\{0.7n,x\}\log\alpha-\log 27&<2.0\cdot 10^{10} k^3(\log k)^2 \log x,\\
			\min\{0.7n,x\}&<5.0\cdot 10^{10} k^3(\log k)^2 \log x.
		\end{align*}
		
		We distinguish between two cases.
		
		\textbf{Case I}: If the  $\min\{0.7n,x\}=0.7n$, then $n<7.2\cdot 10^{10} k^3(\log k)^2 \log x$, and using Lemma \ref{lem3.3}, we get
		\begin{align*}
			n&<7.2\cdot 10^{10} k^3(\log k)^2 \log \left(  2.6\cdot 10^{15} n k^4(\log k)^3\log n \right)  \\
			&=7.2\cdot 10^{10} k^3(\log k)^2 \left(\log   2.6\cdot 10^{15} +\log n +4\log k +3\log\log k+\log\log n \right)\\
			&<7.2\cdot 10^{10} k^3(\log k)^3\log n \left(\dfrac{36}{\log n\log k} +\dfrac{1}{\log k} +\dfrac{4}{\log n} +3\dfrac{\log\log k}{\log n\log k}+\dfrac{\log\log n}{\log n\log k} \right)\\
			&<7.3\cdot 10^{11} k^3(\log k)^3\log n.
		\end{align*}
		This tells us that $n/\log n<7.3\cdot 10^{11} k^3(\log k)^3$.
		Applying Lemma \ref{Guz} with  
		$$
		p:=n, \quad r:=1, \quad T:=7.3\cdot 10^{11} k^3(\log k)^3,
		$$ 
		we get 
		\begin{align*}
			n&<2\cdot7.3\cdot 10^{11} k^3(\log k)^3\log \left(7.3\cdot 10^{11} k^3(\log k)^3\right)\\
			&<1.5\cdot 10^{12} k^3(\log k)^3 \left(\log (7.3\cdot 10^{11})+3\log k+3\log\log k\right)\\
			&<1.5\cdot 10^{12} k^3(\log k)^4 \left(\dfrac{28}{\log k}+3+\dfrac{3\log\log k}{\log k}\right),
		\end{align*}
		so that $n<6.3\cdot 10^{13} k^3(\log k)^4$. Inserting this upper bound of $n$ in Lemma \ref{lem3.3}, we have
		\begin{align*}
			x&<2.6\cdot 10^{15} \left(6.3\cdot 10^{13} k^3(\log k)^4\right) k^4(\log k)^3\log \left(6.3\cdot 10^{13} k^3(\log k)^4\right)\\
			&< 1.8\cdot 10^{29}k^7(\log k)^8\left(\dfrac{32}{\log k}+3+\dfrac{4\log\log k}{\log k}\right)\\
			&<  8.7\cdot 10^{30}k^7(\log k)^8.
		\end{align*}
		By inequality \eqref{m_b}, $m<(n+3)x+3<5.6\cdot 10^{44}k^{10}(\log k)^{12}$.
		
		\medskip
		
		\textbf{Case II}: If the  $\min\{0.7n,x\}=x$, then $x<5.0\cdot 10^{10} k^3(\log k)^2 \log x$. Applying Lemma \ref{Guz} with  $p:=x$, $r:=1$, $T:=5.0\cdot 10^{10} k^3(\log k)^2$, we get 
		\begin{align*}
			x&<2\cdot 5.0\cdot 10^{10} k^3(\log k)^2\log \left(5.0\cdot 10^{10} k^3(\log k)^2\right)\\
			&=1.0\cdot 10^{11} k^3(\log k)^2 \left(\log (5.0\cdot 10^{10})+3\log k+2\log\log k\right)\\
			&< 10^{11} k^3(\log k)^3 \left(\dfrac{25}{\log k}+3+\dfrac{2\log\log k}{\log k}\right),
		\end{align*}
		so that 
		\begin{align}\label{xc}
			x<4.0\cdot 10^{12} k^3(\log k)^3.
		\end{align}
		Now, since $x\le 0.7n$, then for each $i\in\{-1,0,1\}$, the inequality 
		\begin{align*}
			\dfrac{x}{\alpha^{n+i-1}}<\dfrac{0.7n}{\alpha^{n-2}}<\dfrac{1}{\alpha^{0.97n}},
		\end{align*}
		holds for all $n>800$. Now, Lemma \ref{lem:Lnx} tells us that 
		\begin{align}\label{eq3.17}
			\left(L_{n+i}^{(k)}\right)^x = f_k(\alpha)^x (2\alpha-1)^x\alpha^{(n+i-1)x}(1 + \eta_{n}),\qquad\text{with}\qquad	
			|\eta_{n}| < \dfrac{1.5xe^{1.5x/\alpha^{n+i-1}}}{\alpha^{n+i-1}}<\dfrac{2}{\alpha^{0.97n}}.
		\end{align}
		So, we rewrite \eqref{eq:main} using \eqref{eq3.17} as
		\begin{align*}
			f_k(\alpha)(2\alpha-1)\alpha^{m-1}&-f_k(\alpha)^x(2\alpha-1)^x\alpha^{(n-1)x}\left(1+\alpha^{-x}-\alpha^{-2x}\right)\\
			&=f_k(\alpha)^x(2\alpha-1)^x\alpha^{(n-1)x}\left(1+\alpha^{-x}-\alpha^{-2x}\right)\eta_n+e_k(m).
		\end{align*}
		Dividing both sides of the above equality by $	f_k(\alpha)(2\alpha-1)\alpha^{m-1}$ and taking absolute values, we get
		\begin{align*}
			&\left|f_k(\alpha)^{x-1}(2\alpha-1)^{x-1}\alpha^{(n-1)x-(m-1)}\left(1+\alpha^{-x}-\alpha^{-2x}\right)-1\right|\\
			&<|\eta_n|f_k(\alpha)^{x-1}(2\alpha-1)^{x-1}\alpha^{(n-1)x-(m-1)}\left(1+\alpha^{-x}-\alpha^{-2x}\right)+\dfrac{1.5}{f_k(\alpha)(2\alpha-1)\alpha^{m-1}}\\
			&<\dfrac{2}{\alpha^{0.97n}}\cdot \dfrac{1}{\left(f_k(\alpha)(2\alpha-1)\alpha\right)^{1-x}}\cdot  \dfrac{1}{\alpha^{(m-2)-(n-2)x}}
			+\dfrac{1.5}{\alpha^{m-1}}\\
			&<\dfrac{2}{\alpha^{0.97n}}\cdot \dfrac{1/\left(0.5\cdot 2\cdot1\right)^{1-x}}{\alpha^{-5}}
			+\dfrac{1.5}{\alpha^{n-4}}<\dfrac{2}{\alpha^{0.97n-5}} +\dfrac{1.5}{\alpha^{n-4}}\\
			& <\dfrac{3}{\alpha^{0.97n-5}},
		\end{align*}
		for $n>800$, $x\ge 2$ and $k\ge 2$. This means that
		\begin{align}\label{g5}
			|\Gamma_5|:=	\left|f_k(\alpha)^{x-1}(2\alpha-1)^{x-1}\alpha^{(n-1)x-(m-1)}\left(1+\alpha^{-x}-\alpha^{-2x}\right)-1\right|<\dfrac{3}{\alpha^{0.97n-5}}.
		\end{align}
		Note that $\Gamma_5\ne 0$, otherwise we would have 
		\begin{align*}
			\left((2\alpha-1)f_k(\alpha)\right)^{x-1}\alpha^{(n-1)x-(m-1)}\left(1+\alpha^{-x}-\alpha^{-2x}\right)=1.
		\end{align*}	
		Taking norms in ${\mathbb K}={\mathbb Q}(\alpha)$ as before and using parts $(i)$ and $(ii)$ of Lemma \ref{lemGLm}, the above equation becomes
		\begin{align}\label{prod}
			\left(\dfrac{\left( 2^{k+1} - 3\right)  (k - 1)^2}{ 2^{k+1}k^k - (k + 1)^{k+1}}\right)^{x-1}\cdot N\left(1+\alpha^{-x}-\alpha^{-2x}\right)&=1,\nonumber\\
			\left(\dfrac{\left( 2^{k+1} - 3\right)  (k - 1)^2}{ 2^{k+1}k^k - (k + 1)^{k+1}}\right)^{x-1}\cdot \left(1+\alpha^{-x}-\alpha^{-2x}\right)\prod_{j= 2}^k\left|1+(\alpha^{(j)})^{-x}-(\alpha^{(j)})^{-2x}\right|&=1.
		\end{align}
		Since $1+\alpha^{-x}-\alpha^{-2x}<2$, for all $x\ge 2$, we have from \eqref{prod} that
		\begin{align*}
			1&=	\left(\dfrac{\left( 2^{k+1} - 3\right)  (k - 1)^2}{ 2^{k+1}k^k - (k + 1)^{k+1}}\right)^{x-1}\cdot \left(1+\alpha^{-x}-\alpha^{-2x}\right)\prod_{j= 2}^k\left|1+(\alpha^{(j)})^{-x}-(\alpha^{(j)})^{-2x}\right|\\
			&<\left(\dfrac{\left( 2^{k+1} - 3\right)  (k - 1)^2}{ 2^{k+1}k^k - (k + 1)^{k+1}}\right)^{x-1}\cdot 2\cdot \prod_{j= 2}^k\left|(\alpha^{(j)})^{-2x}-(\alpha^{(j)})^{-x}-1\right|\\
			&<\left(\dfrac{\left( 2^{k+1} - 3\right)  (k - 1)^2}{ 2^{k+1}k^k - (k + 1)^{k+1}}\right)^{x-1}\cdot 2\prod_{j= 2}^k3\left|(\alpha^{(j)})^{-2x}\right|
			<\left(\dfrac{\left( 2^{k+1} - 3\right)  (k - 1)^2}{ 2^{k+1}k^k - (k + 1)^{k+1}}\right)^{x-1}\cdot 2\cdot 3^{k-1}\alpha^{2x}\\
			&<\left(8^{x/(x-1)}\cdot3^{(k-1)/(x-1)} \dfrac{\left( 2^{k+1} - 3\right)  (k - 1)^2}{ 2^{k+1}k^k - (k + 1)^{k+1}}\right)^{x-1}
			\\
			&<\left(8^{2}\cdot3^{k-1} \dfrac{\left( 2^{k+1} - 3\right)  (k - 1)^2}{ 2^{k+1}k^k - (k + 1)^{k+1}}\right)^{x-1}
			<1,
		\end{align*}
		for all $k\ge 8$ and $x\ge 2$, a contradiction. When $k=3,4,5,6,7$, we have that the expression 
		\begin{align*}
			\dfrac{\left( 2^{k+1} - 3\right)  (k - 1)^2}{ 2^{k+1}k^k - (k + 1)^{k+1}},	
		\end{align*}
		in \eqref{prod} is equal to 
		$$
		13/44,\quad  29/563, \qquad 61/9584, \quad 125/205937\quad  253/5390272,
		$$
		respectively. Since $13\nmid 44$, $29\nmid 563$, $61\nmid 9584$, $125\nmid 205937$ and $253\nmid 5390272$, equation \eqref{prod} can not hold when $k=3,4,5,6,7$. When $k=2$, we have that  
		\begin{align*}
			\dfrac{\left( 2^{k+1} - 3\right)  (k - 1)^2}{ 2^{k+1}k^k - (k + 1)^{k+1}}=1.	
		\end{align*}
		In this case, equation \eqref{prod} becomes 
		\begin{align*}
			\left(1+\alpha^{-x}-\alpha^{-2x}\right)\left(1+\beta^{-x}-\beta^{-2x}\right)=\pm 1,
		\end{align*}
		where $\alpha=\left(1+\sqrt5\right)/2$ and $\beta=\left(1-\sqrt5\right)/2 $. We multiply both sides of the above equation by $(\alpha\beta)^{2x}=1$ and get 
		\begin{align}\label{lx}
			\left(\alpha^{2x}+\alpha^{x}-1\right)\left(\beta^{2x}+\beta^{x}-1\right)&=\pm 1,\nonumber\\
			1+ (-1)^x(\alpha^x+\beta^x)-\left(\alpha^{2x}+\beta^{2x}\right)+(-1)^x+1-(\alpha^x+\beta^x)      &=\pm 1,\nonumber\\
			2+ ((-1)^x-1)L_x-L_{2x}+(-1)^x     &=\pm 1.
		\end{align}
		Now, if $x$ is even in \eqref{lx}, then we have $L_{2x} =2,4$ leading to $x=0$ which contradicts the working assumption that $x\ge 2$. On the other hand, if $x$ is odd in \eqref{lx}, then $L_{2x}+L_x =0,2$, therefore we get that $L_{2x} =-L_x,~2-L_x$, which are both negative numbers for $x\ge 2$, another contradiction. Thus, in all cases, $\Gamma_5\ne 0$.
		
		The algebraic number field containing the following $\gamma_i$'s is $\mathbb{K} := \mathbb{Q}(\alpha)$. We have $D = k$, $t :=3$,
		\begin{equation}\nonumber
			\begin{aligned}
				\gamma_{1}&:=(2\alpha-1)f_k(\alpha),\qquad \gamma_{2}:=\alpha,\qquad\gamma_{3}:=\left(1+\alpha^{-x}-\alpha^{-2x}\right),\\
				b_{1}&:=x-1,\qquad b_{2}:=(n-1)x-(m-1),\qquad b_3:=1.
			\end{aligned}
		\end{equation}
		As before, we take $A_{1}:=6k\log k$ and $A_{2}:=0.7$. 
		For $A_3$, we first compute
		\begin{align*}
			Dh(\gamma_3)=kh\left(1+\alpha^{-x}-\alpha^{-2x}\right)<3kx(\log\alpha)/k,
		\end{align*}
		so that we can take $A_3:=3x$.
		By the second inequality in \eqref{m_b}, we can take $B:=x$. So, Theorem \ref{thm:Mat} gives
		\begin{align}\label{eq3.19}
			\log |\Gamma_5| &> -1.4\cdot 30^{6} \cdot 3^{4.5}\cdot k^2 (1+\log k)(1+\log x)\cdot   0.7\cdot 6k\log k\cdot 3x\nonumber\\
			&> -1.1\cdot 10^{13} x k^3(\log k)^2 \log x.
		\end{align}
		Comparing \eqref{g5} and \eqref{eq3.19}, we get
		\begin{align*}
			(0.97n-5)\log\alpha-\log 3&<1.1\cdot 10^{13} x k^3(\log k)^2 \log x,\\
			n&<2.8\cdot 10^{13}  k^3(\log k)^2 x\log x.
		\end{align*}
		Moreover, by \eqref{xc}, $x<4.0\cdot 10^{12} k^3(\log k)^3$, so we have 
		\begin{align*}
			n&<2.8\cdot 10^{13}  k^3(\log k)^2 \cdot 4.0\cdot 10^{12} k^3(\log k)^3\cdot \log \left(4.0\cdot 10^{12} k^3(\log k)^3\right)\\
			&<1.2\cdot 10^{26} k^6(\log k)^6\left(\dfrac{30}{\log k}+3+3\dfrac{\log\log k}{\log k}\right)\\
			&<5.5\cdot 10^{27} k^6(\log k)^6.
		\end{align*}
		Lastly, by inequality \eqref{m_b}, $m<(n+3)x+3<2.3\cdot 10^{40}k^9(\log k)^{9}$.
		
		Comparing all inequalities of $n$, $m$ and $x$ in both cases, we have the inequalities in Lemma \ref{lem3.4}.
	\end{proof}
	\subsubsection{The case $k>800$}
	The inequalities in Lemma \ref{lem3.4} were obtained under the assumptions that $n > 800$. However, when $n \le 800$, the inequalities \eqref{m_b} and \eqref{x_bound1} yield even smaller
	upper bounds for $x$ and $m$ in terms of $k$.
	From now on, let us assume that $k > 800$.
	By Lemma \ref{lem3.4}, we get 
	\begin{align*}
		n+i<5.51\cdot 10^{27} k^6(\log k)^6<2^{0.24k}	\qquad\text{and}\qquad m<5.6\cdot 10^{44}k^{10}(\log k)^{12} <2^{0.39k},
	\end{align*}
	for all $k>800$ and $i\in\{-1,0,1\}$. 
	By Lemma \ref{Ln:x} with $c=0.39$ and $x=1$, we have
	\begin{align*}
		L_m^{(k)} = 3 \cdot 2^{m-2} \left( 1 + \xi_m \right), \quad \text{with} \quad |\xi_m| < \dfrac{2}{2^{0.22k}}.
	\end{align*}
	Again, by Lemma \ref{Ln:x} with $c=0.24$ and $x\ge 2$, we have
	\begin{align*}
		\left(L_{n+i}^{(k)}\right)^x = 3^x\cdot 2^{(n+i-2)x} \left(1 +\xi_{n+i} \right), \quad \text{with} \quad |\xi_{n+i}| <\dfrac{2}{2^{0.52k}}.	
	\end{align*}
Now, for all $n>k\ge 2$, we rewrite \eqref{eq:main} as
	\begin{align*}
		\left|3^x\cdot 2^{(n-1)x}\left(1+2^{-x}-2^{-2x}\right) - 3 \cdot 2^{m-2} \right|&\le\left( 3^x\cdot 2^{(n-1)x} \left(1+2^{-x}+2^{-2x}\right)+3 \cdot 2^{m-2} \right)\cdot \dfrac{2}{2^{0.22k}},
	\end{align*}
where we chose $2/2^{0.22k}$ over $2/2^{0.52k}$ since $2/2^{0.52k}<2/2^{0.22k}$. In the above, we divide through by the $\max\{3^x\cdot 2^{(n-1)x}\left(1+2^{-x}-2^{-2x}\right),~ 3 \cdot 2^{m-2}\}$. In the left--hand side, we have 
	$$\left|3^{x-1}\cdot2^{-(m-2)}\cdot 2^{(n-1)x}\left(1+2^{-x}-2^{-2x}\right) -1\right|\quad\text{or}\quad
	\left|3^{1-x}\cdot2^{(m-2)}\cdot 2^{-(n-1)x}\left(1+2^{-x}-2^{-2x}\right)^{-1} -1\right|,$$
	while on the right--hand side, we have
	$$
	\left(\dfrac{1+2^{-x}+2^{-2x}}{1+2^{-x}-2^{-2x}}+\dfrac{3 \cdot 2^{m-2}}{3^x\cdot 2^{(n-1)x}\left(1+2^{-x}-2^{-2x}\right)} \right)\cdot \dfrac{2}{2^{0.22k}}<(1.25+1)\cdot \dfrac{2}{2^{0.22k}}= \dfrac{4.5}{2^{0.22k}},
	$$
	or
	$$
	\left(\dfrac{3^x\cdot 2^{(n-1)x}(1+2^{-x}+2^{-2x})}{3 \cdot 2^{m-2}}+1 \right)\cdot \dfrac{2}{2^{0.22k}}\le(1+1)\cdot \dfrac{2}{2^{0.22k}}= \dfrac{4}{2^{0.22k}},
	$$
	respectively. Thus, in both cases, we have
	\begin{align}\label{eqc1}
		\left|3^{x-1}\cdot2^{-(m-2)}\cdot 2^{(n-1)x}\left(1+2^{-x}-2^{-2x}\right) -1\right|< \dfrac{4.5}{2^{0.22k}}.
	\end{align}
	Again, putting $y:=m-2-(n-3)x$ in \eqref{eqc1} and dividing through  by $2^y$, we get
	\begin{align}\label{eq:01}
		\left|\frac{3^{x-1} (2^{2x}+2^x-1)}{2^y}-1\right|<\dfrac{4.5}{2^{0.22k}}.
	\end{align}
	Again, the left--hand side in \eqref{eq:01} is 
	\begin{align*}
		\left|e^{(x-1)\log 3+\log(2^{2x}+2^x+1)-y\log 2}-1\right|,
	\end{align*}
	and the right--hand side is $<1/2$. So, we get by a simple argument in calculus that
	\begin{align*}
		\left|(x-1)\log 3+\log(2^{2x}+2^x-1)-y\log 2\right|<\dfrac{9}{2^{0.22k}}.
	\end{align*}
	This can be rewritten as 
	\begin{align*}
		\left|(x-1)\log 3+2x\log 2+\log\left(1+\frac{2^x-1}{2^{2x}}\right)-y\log 2\right|<\dfrac{9}{2^{0.22k}},
	\end{align*}
	or 
	\begin{equation}\label{eq:11}
		|(x-1)\log 3-(y-2x)\log 2|<\dfrac{9}{2^{0.22k}}+\log\left(1+\frac{2^x+1}{2^{2x}}\right)<\dfrac{9}{2^{0.22k}}+\frac{2^x-1}{2^{2x}}.
	\end{equation}
	Since $x\ge 2$,  the right--hand side in \eqref{eq:11} is still $<6/16$, so $z:=y-2x>0$ as before.  This means that the conclusion in \eqref{eq:2} still holds even in this case.
	
	The left--hand side in \eqref{eq:11} is of the form 
	\begin{align*}
		|b_1\log \gamma_1-b_2\log \gamma_2|,
	\end{align*}
	where $\gamma_1:=3,~\gamma_2:=2,~b_1:=x-1$ and $b_2:=z$. So, we take as before that $D=1$, $\log A_1=\log 3$, $\log A_2=1$ and
	$b'<3x$, by \eqref{eq:2}. Now, Theorem \ref{thm:LMNh} shows that 
	\begin{align}\label{eq:b1}
		|(x-1)\log 3-(y-2x)\log 2|>\exp\left(-24.34 \left(\max\left\{21, \log b'+0.14\right\}\right)^2\log 3\right).
	\end{align}
	Since 
	\begin{align*}
		\frac{9}{2^{0.22k}}+\frac{2^x-1}{2^{2x}}\le \frac{9}{2^{0.22k}}+\frac{1}{2^x}\le \frac{10}{2^{\min\{0.22k,x\}}},
	\end{align*}
	we compare \eqref{eq:11} and \eqref{eq:b1} to conclude that
	\begin{equation}
		\label{eq:31}
		\min \{0.22 k,x\} \log 2-\log 10<24.34 (\max\{21,\log(3x)+0.14\})^2\log 3.
	\end{equation}
	We proceed in two cases, when $\min \{0.22 k,x\}$ is $0.22k$ or $x$.
	\begin{enumerate}[(i)]
		\item Assume first that $\min \{0.22 k,x\}=0.22k$, then
		$$
		(0.22 \log 2)k<24.34( \max\{21, \log(3x)+0.14\})^2\log 3+\log 10.
		$$
		If $\max\{21, \log(3x)+0.14\}=21$, we get 
		$
		(0.22\log 2)k<24.34\cdot (21)^2\log 3+\log 260$, so $k<77350$.
		If $\max\{21, \log(3x)+0.14\}= \log(3x)+0.14$, we get 
		$$
		(0.22\log 2)k <24.34\log 3 (\log(3x)+0.14)^2+\log 10,
		$$
		so that
		$$
		k<176 (\log(3x)+0.14)^2+16<187(\log(3x)+0.14))^2,
		$$
		for all $x\ge 1$. By Lemma \ref{lem3.4}, we have 
		$$
		x<8.7\cdot 10^{30} k^7 (\log k)^8.
		$$
		Thus, 
		\begin{align*}
			x  &<  8.7\cdot 10^{30} (187(\log(3x)+0.14)^2)^{7} (\log 187+2\log(\log(3x)+0.14))^8.
		\end{align*}
		We get $x<10^{89}$.
		\item Next, if $\min \{0.22 k,x\}=x$, we then get
		$$
		x\log 2-\log 10<24.34(\max\{21,\log(3x)+0.10\})^2\log 3.
		$$
		If the maximum in the right above is in $21$, we then get $x<18000$, while if the maximum is in $\log(3x)+0.14$, we then get
		$$
		x\log 2<24.34\log 3 (\log(3x)+0.14)^2+\log 10,
		$$
		so that $x<4000$.
	\end{enumerate} 
	So, in all cases $x<10^{89}$. We now return to \eqref{eq:11} which we write as 
	\begin{align*}
		\left|\frac{\log 3}{\log 2}-\frac{y-2x}{x-1}\right|<\frac{9/2^{0.22k}+(2^x-1)/(2^{2x})}{(x-1)\log 2}<\frac{13/2^{0.22k}+2(2^x-1)/(2^{2x})}{x-1}.
	\end{align*}
	Now, we generated the $187$th convergent $p_{187}/q_{187}=[a_0,a_1,\ldots,a_{187}]=[1,1,\ldots,2]$ of $\log 3/\log 2$ and we got that $q_{187}>10^{89}$ and that  $\max\{a_k: 0\le k\le 187\}=55$. By well--known properties of continued fractions, we get 
	\begin{align*}
		\frac{1}{57(x-1)^2}<\left|\frac{\log 3}{\log 2}-\frac{y-2x}{x-1}\right|<\frac{13/2^{0.22k}+2(2^x-1)/(2^{2x})}{x-1}.
	\end{align*}
	Therefore, we have
	\begin{align*}
		\frac{1}{57(x-1)}<\frac{13}{2^{0.22k}}+\frac{2^x-1}{2^{2x-1}}.
	\end{align*}
	The left--hand side is at least $1/(57\cdot 10^{89})$, while $13/2^{0.22k}<10^{-50}$ since $k>800$. So, we get
	$$
	\frac{1}{58 (x-1)}<\frac{2^x-1}{2^{2x-1}},
	$$
	which gives $x\le 10$, as before in Subsection \ref{sub322}. Thus,
	\begin{equation}
		\label{eq:41}
		|(x-1)\log 3-(y-2x)\log 2|<\frac{9}{2^{0.22k}}+\frac{2^x-1}{2^{2x}}.
	\end{equation}
	Again, we used a simple code to compute the minimum values of the expression 
	\[
	|(x-1)\log 3 - (y-2x)\log 2|,
	\]
	for integer values of \( y - 2x \). As before, by assigning values to \( x \) ranging from 2 to 10, the code found the \( y \) that minimized the expression for each \( x \). The results still showed that the left--hand side of \eqref{eq:41} was at least 
	\[
	0.28,~0.11,~0.16,~0.23,~0.05,~0.33,~0.06,~0.22,~0.18,
	\]
	for \( x = 2, 3, \ldots, 10 \), respectively. Again, it turned out that for no $x\in \{2,3,\ldots,10\}$, inequality \eqref{eq:41} is satisfied. This finishes the argument that the case $k>800$ is not possible here. 
	
	\subsubsection{The case $k\le800$}
	To proceed, let us assume that $n > 800$, so, we can use Lemma \ref{lem3.4} to obtain upper bounds on $n$, $x$, and $m$. We assume that $x \leq 10$. Using inequality \eqref{g5}, we write $|\Lambda_5|:=|\log (\Gamma_5+1)|$ as 
	\begin{align*}
		|\Lambda_5|&:=	\left|(x-1)\log f_k(\alpha)+(x-1)\log(2\alpha-1)+\left((n-1)x-(m-1)\right)\log\alpha+\log\left(1+\alpha^{-x}-\alpha^{-2x}\right)\right|\\
		&<\dfrac{4.5}{\alpha^{0.97n-5}},
	\end{align*}
	where $a_1 := x-1$, $a_2 := x-1$, $a_3 := (n-1)x-(m-1)$, $a_4 := 1$ are integers with 
	\[
	\max\{|a_i| : 1 \leq i \leq 4\}  < m<5.6\cdot 10^{44}k^{10}(\log k)^{12}<4.8\cdot 10^{83},
	\]
	where we used Lemma \ref{lem3.4}, for all $k\le 800$.
	
	For each $k \in [2, 800]$ and $x\in [2,10]$, we used the LLL--algorithm to compute a lower bound for the smallest nonzero number of the form $|\Lambda_5|$, with integer coefficients $a_i$ not exceeding $5.6\cdot 10^{44}k^{10}(\log k)^{12}$ in absolute value. Specifically, we consider the approximation lattice
	$$ \mathcal{A}=\begin{pmatrix}
		1 & 0 & 0 & 0 \\
		0 & 1 & 0 & 0 \\
		0 & 0 & 1 & 0 \\
		\lfloor C\log f_k(\alpha)\rfloor & \lfloor C\log (2\alpha-1) \rfloor& \lfloor C\log \alpha \rfloor& \lfloor C\log \left(1+\alpha^{-x}-\alpha^{-2x}\right) \rfloor 
	\end{pmatrix} ,$$
	with $C:= 10^{335}$ and choose $y:=\left(0,0,0,0\right)$. Now, by Lemma \ref{lem2.5}, we get 
	$$c_2:=10^{-85}\qquad\text{and}\qquad l\left(\mathcal{L},y\right)\ge c_1:=1.56\cdot10^{84}.$$
	So, Lemma \ref{lem2.6} gives $S=9.3\cdot 10^{167}$ and $T=9.6\cdot 10^{83}$. Since $c_1^2\ge T^2+S$, then choosing $c_3:=4.5$ and $c_4:=\log\alpha$, we get $0.97n-5\le 1235$, or $n\le 1278$.
	
	Now, we work with $x>10$. We go to inequality \eqref{g4} and write $|\Lambda_4|:=|\log (\Gamma_4+1)|$ as 
	\begin{align*}
		|\Lambda_4|:=	\left|\left(m-1-(n-1)x\right)\log\alpha-(x-1)\log \left((2\alpha-1)f_k(\alpha)\right)\right|<\dfrac{40.5}{\alpha^{\min\{0.7n,x\}}}.
	\end{align*}
	where we have used the fact that $\min\{0.7n,x\} \geq 10$. Dividing both sides of the above inequality by $(x - 1) \log \alpha$, we obtain
	\begin{align}\label{eqr2}	
		\left| \frac{\log \left((2\alpha-1)f_k(\alpha)\right)}{\log \alpha} - \frac{m-1-(n-1)x}{x - 1} \right| < \frac{40.5}{\alpha^{\min\{0.7n,x\}}(x - 1) \log \alpha} < \frac{10}{\alpha^{\min\{0.7n,x\}}},
	\end{align}
	for all $x>10$. We proceed by examining two cases from \eqref{eqr2}.
	
	\textbf{Case I}: If $m-1 =  (n-1)x$, then inequality \eqref{eqr2} becomes
	\[
	\left| \frac{\log \left((2\alpha-1)f_k(\alpha)\right)}{\log \alpha} \right|  < \frac{10}{\alpha^{\min\{0.7n,x\}}}.
	\]
	We get 
	\begin{align*}
		\frac{10}{\alpha^{\min\{0.7n,x\}}}> \frac{\log \left((2\alpha-1)f_k(\alpha)\right)}{\log \alpha}>\frac{\log (2.2\cdot 0.5)}{\log 2}>0.1
	\end{align*}
	for all $k \in [2, 800]$. Thus, $\alpha^{\min\{0.7n,x\}}<100$, or equivalently, $\min\{0.7n,x\}<10$. This contradicts our assumptions that $n > 800$ and $x > 10$. So this case is not possible.
	
	\textbf{Case II}: If $m-1 \ne  (n-1)x$, we have
	\begin{align*}
		|\Lambda_4|:=	\left|(x-1)\log f_k(\alpha)+(x-1)\log(2\alpha-1)+\left((n-1)x-(m-1)\right)\log\alpha\right|<\dfrac{40.5}{\alpha^{\min\{0.7n,x\}}},
	\end{align*}
	where $a_1 := x-1$, $a_2 := x-1$, $a_3 := (n-1)x-(m-1)$ are integers with 
	\[
	\max\{|a_i| : 1 \leq i \leq 3\}  < m<5.6\cdot 10^{44}k^{10}(\log k)^{12}<4.8\cdot 10^{83},
	\]	
	For each $k \in [2, 800]$, we again used the LLL--algorithm to compute a lower bound for the smallest nonzero number of the form $|\Lambda_4|$, with integer coefficients $a_i$ not exceeding $5.6\cdot 10^{44}k^{10}(\log k)^{12}$ in absolute value. We consider the approximation lattice
	$$ \mathcal{A}=\begin{pmatrix}
		1 & 0 & 0 \\
		0 & 1 & 0 \\
		\lfloor  C\log f_k(\alpha)\rfloor & \lfloor C\log (2\alpha-1) \rfloor& \lfloor C\log \alpha \rfloor
	\end{pmatrix} ,$$
	with $C:= 10^{252}$ and choose $y:=\left(0,0,0\right)$. So, by Lemma \ref{lem2.5}, we get 
	$$c_2=10^{-87}\qquad\text{and}\qquad l\left(\mathcal{L},y\right)\ge c_1:=1.2\cdot10^{85}.$$
	So, Lemma \ref{lem2.6} gives $S=7.0\cdot 10^{167}$ and $T=7.2\cdot 10^{83}$. Since $c_1^2\ge T^2+S$, then choosing $c_3:=40.5$ and $c_4:=\log\alpha$, we get $\min\{0.7n,x\}\le 825$. This implies that $n\le 1178$ or $x\le 825$.
	
	Let us first assume that $n\le 1178$ holds, that is $\min\{0.7n,x\}=0.7n$. We go to inequality \eqref{g3} and write $|\Lambda_3|:=|\log (\Gamma_3+1)|$ as 
	\begin{align*}
		|\Lambda_3|:=\left|\log f_k(\alpha)+\log (2\alpha-1)+(m-1)\log\alpha-x\log\left(L_{n+1}^{(k)}\right)\right|<\dfrac{1.5}{1.4^x},
	\end{align*}
	where $a_1 := 1$, $a_2 := 1$, $a_3 := m-1$, $a_4 := -x$ are integers with 
	\[
	\max\{|a_i| : 1 \leq i \leq 4\}  < m<5.6\cdot 10^{44}k^{10}(\log k)^{12}<4.8\cdot 10^{83},
	\]
	where we used Lemma \ref{lem3.4}, for all $k\le 800$.
	
	For each $k \in [2, 800]$ and $n\in [801,1178]$, we used the LLL--algorithm to compute a lower bound for the smallest nonzero number of the form $|\Lambda_3|$, with integer coefficients $a_i$ not exceeding $5.6\cdot 10^{44}k^{10}(\log k)^{12}$ in absolute value. Specifically, we consider the approximation lattice
	$$ \mathcal{A}=\begin{pmatrix}
		1 & 0 & 0 & 0 \\
		0 & 1 & 0 & 0 \\
		0 & 0 & 1 & 0 \\
		\lfloor  C\log f_k(\alpha)\rfloor & \lfloor C\log (2\alpha-1) \rfloor& \lfloor C\log \alpha \rfloor& \left\lfloor C\log \left(L_{n+1}^{(k)}\right) \right\rfloor 
	\end{pmatrix} ,$$
	with $C:= 10^{335}$, $y:=\left(0,0,0,0\right)$, $c_2:=10^{-85}$ and $l\left(\mathcal{L},y\right)\ge c_1:=1.56\cdot10^{84}$.
	As before with $\Lambda_5$, Lemma \ref{lem2.6} gives $S=9.3\cdot 10^{167}$ and $T=9.6\cdot 10^{83}$. Choosing $c_3:=1.5$ and $c_4:=\log 1.4$, we get $x\le 1722$.
	
	Next, if $x\le 825$ holds, that is $\min\{0.7n,x\}=x$, we go back to $|\Lambda_5|$ and apply LLL--algorithm. We obtain a similar conclusion as before that $n\le 1278$. To sum up, we have that the integer solutions $n$, $m$, $k$, $x$ to \eqref{eq:main} with $n > k \ge 2$, $k \le 800$, $n > 800$ and $x \ge 2$, satisfy $n \le 1278$ and $x \le 1722$.
	
	Finally, we consider the case when $n \le 800$ and since we are working with $k \le 800$ and $n > k$, we get $k \in [2,799]$. Now, we use $\Lambda_3$ as defined before. We apply the LLL-- algorithm with 
	\[
	\max\{|a_i| : 1 \leq i \leq 4\}  < m<(n+3)x+3<803\cdot 2.6\cdot 10^{15} n k^4(\log k)^3\log n+3<1.4\cdot 10^{36},
	\]
	where we have used Lemma \ref{lem3.3}. We consider the same approximation lattice as for $\Lambda_3$, but
	with $C:= 10^{145}$. Choosing $y:=\left(0,0,0,0\right)$, we get $c_2:=10^{-39}$ and $l\left(\mathcal{L},y\right)\ge c_1:=4.0\cdot10^{37}$.
	Lemma \ref{lem2.6} gives $S=7.84\cdot 10^{72}$ and $T=2.8\cdot 10^{36}$. Choosing $c_3:=1.5$ and $c_4:=\log 1.4$, we get $x\le 736$.
	
	To conclude the case $n>k$, we look for solutions to \eqref{eq:main} in the ranges  $k\in [2, 800]$, $n \in [k + 1, 800]$ and $x \in [2, 736]$ or $n \in [801, 1278]$ and $x \in [2, 1722]$, all with $nx-3 <m<(n+3)x+3$. A computational search in SageMath 9.5 allows us to conclude that there are no other integral solutions to \eqref{eq:main} apart from those given in Theorem \ref{thm1.1}. To efficiently handle large \( L_n^{(k)} \) values, our SageMath 9.5 code utilized batch processing, iterating through all \( (n, m,k, x) \) combinations within these ranges. \qed

	\section*{Acknowledgments} 
	We thank the school of mathematics at Stellenbosch university for providing the necessary resources and support during this research. Additionally, we are grateful for the use of virtual SageMath 9.5 libraries that made computations in this work possible.

	\section*{Addresses}
	
	$ ^{1} $ Mathematics Division, Stellenbosch University, Stellenbosch, South Africa.
	
	Email: \url{hbatte91@gmail.com} \qquad\url{https://orcid.org/0000-0003-3882-0189}
	\vspace{0.3cm}\\
	\noindent 
	$ ^{2} $ Centro de Ciencias Matem\'aticas UNAM, Morelia, Mexico.
	
	Email: \url{fluca@sun.ac.za}\qquad \url{https://orcid.org/0000-0003-1321-4422}
	
\end{document}